\documentclass{imsart}

\RequirePackage{amsthm}
\RequirePackage{enumerate}
\RequirePackage{color}
\RequirePackage{placeins}
\RequirePackage{pdfsync}
\RequirePackage{mathbbol}
\RequirePackage{natbib}
\RequirePackage{amsmath}
\RequirePackage{amsthm}
\RequirePackage[colorlinks]{hyperref}
\RequirePackage{xcolor}
\RequirePackage{array}
\RequirePackage{booktabs}
\RequirePackage{graphicx}
\definecolor{halfgray}{gray}{0.55} 
\definecolor{webgreen}{rgb}{0,.5,0}
\definecolor{webbrown}{rgb}{.6,0,0}
\definecolor{RoyalBlue}{rgb}{0,0.08,0.45}
\RequirePackage{hypernat}
\RequirePackage{multirow}

\hypersetup{%
   colorlinks=true, linktocpage=true, pdfstartpage=3, pdfstartview=FitV,%
   breaklinks=true, pdfpagemode=UseNone, pageanchor=true, pdfpagemode=UseOutlines,%
   plainpages=false, bookmarksnumbered, bookmarksopen=true, bookmarksopenlevel=1,%
   hypertexnames=true, pdfhighlight=/O,
   urlcolor=webbrown, linkcolor=RoyalBlue, citecolor=webgreen, 
   pdftitle={Tail index estimation, concentration and adaptivity},%
   pdfauthor={\textcopyright\ MT SB,  Paris-Diderot, Math\'ematiques},%
   pdfsubject={},%
   pdfkeywords={},%
   pdfcreator={pdfLaTeX},%
   pdfproducer={LaTeX}%
}   

\numberwithin{equation}{section}
\newtheorem{thm}[equation]{Theorem}
\newtheorem{cor}[equation]{Corollary}
\newtheorem{lem}[equation]{Lemma}
\newtheorem{prp}[equation]{Proposition}

\theoremstyle{remark}

\newtheorem{dfn}[equation]{Definition}
\newtheorem{rem}[equation]{Remark}

\newcommand{\mathd}{\mathrm{d}}
\newcommand{\mathe}{\mathrm{e}}

\newcommand{\EXP}{\ensuremath{\mathbb{E}}}

\newcommand{\var}{\operatorname{Var}}
\newcommand{\esp}{\mathbb{E}}
\newcommand{\pt}{\enspace .}
\newcommand{\ent}{\operatorname{Ent}}
\newcommand{\dis}{\stackrel{d}{=}}

\newcommand{\cov}{\operatorname{Cov}}
\newcommand{\arsinh}{\operatorname{arsinh}}

\arxiv{math.ST/1503.05077}
\begin{document}

\begin{frontmatter}
\title{Tail index estimation, concentration and adaptivity}
\runtitle{Adaptive Hill estimation}
\thankstext{}{Research was partially supported by the ANR14-CE20-0006-01 project AMERISKA network}

\begin{aug}
  \author{\fnms{St\'ephane}  \snm{Boucheron}\ead[label=e1]{stephane.boucheron@univ-paris-diderot.fr}},

  \address{DMA -- CNRS UMR 8553, ENS Ulm\\ LPMA -- CNRS UMR 7599, Sorbonne Paris Cit\'e, Universit\'e Paris-Diderot\\
\printead{e1}}

  \and
  \author{\fnms{Maud} \snm{Thomas}\ead[label=e2]{maud.thomas@univ-paris-diderot.fr}}

\address{LPMA -- CNRS UMR 7599, Sorbonne Paris Cit\'e, Universit\'e Paris-Diderot\\ 
          \printead{e2}}

\runauthor{S. Boucheron and M. Thomas}

 \end{aug}

\begin{abstract}
This paper presents an adaptive version of the Hill estimator based on Lespki's model selection method. This simple data-driven index selection method is shown to satisfy an oracle inequality and is checked to achieve the lower bound recently derived by \citeauthor{carpentierkim2014}. In order to establish the oracle inequality, we derive non-asymptotic variance bounds and concentration inequalities for Hill estimators. These concentration inequalities are derived  from Talagrand's concentration inequality for smooth functions of independent exponentially distributed random variables combined with three tools of Extreme Value Theory: the quantile transform, Karamata's representation of slowly varying functions, and R\'enyi's characterisation for the order statistics of exponential samples. The performance of this computationally and conceptually simple method is illustrated using Monte-Carlo simulations. 
\end{abstract}



\end{frontmatter}


\section{Introduction}
\label{sec:introduction}

The basic questions faced by Extreme Value Analysis consist in estimating the probability of exceeding a threshold that is larger than the sample maximum and estimating a quantile of an order that is larger than 1 minus the reciprocal of the sample size. In words, they consist in making inferences on regions that lie outside the support of the empirical distribution. In order to face these challenges in a sensible framework,  
Extreme Value Theory (\textsf{EVT}) assumes that the sampling distribution $F$ satisfies a regularity condition. Indeed, in heavy-tail analysis, the tail function $\overline{F}=1-F$ is supposed to be regularly varying that is, $\lim_{\tau \to \infty}\overline{F}(\tau x)/\overline{F}(\tau)$ exists for all $x>0$. This amounts to assume the  existence of  some $\gamma>0$ such that the limit is $x^{-1/\gamma}$ for all $x$. In other words, if we define the \emph{excess distribution above the threshold $\tau$} by its survival function: $x \mapsto \overline{F}_\tau(x)=  \overline{F}(x)/\overline{F}(\tau)$ for $x\geq \tau$, then $\overline{F}$ is regularly varying if and only if $F_\tau$ converges weakly towards a Pareto distribution. The sampling distribution $F$ is then said to belong to the \emph{max-domain of attraction} of a Fr\'echet distribution with index $\gamma >0$ (abbreviated in $F \in\textsf{MDA}(\gamma)$) and $\gamma$ is called the \emph{extreme value index}.

The main impediment to large exceedance and large quantile estimation problems alluded above turns out to be  the estimation of the extreme value index. Since the inception of Extreme Value Analysis, many estimators have been defined, analysed and implemented into software. \citet{hill1975} introduced a simple, yet remarkable, collection of estimators: for $k< n$,
\begin{displaymath}
  \widehat{\gamma}(k) = \frac{1}{k} \sum_{i=1}^k \ln
  \frac{X_{(i)}}{X_{(k+1)}}  =  \frac{1}{k} \sum_{i=1}^k  i \ln 
  \frac{X_{(i)}}{X_{(i+1)}}
\end{displaymath}
where $X_{(1)} \geq \ldots \geq X_{(n)}$  are the\emph{ order statistics} of the sample $X_1,\ldots,X_n$ (the non-increasing rearrangement of the sample). 

An integer sequence $(k_n)$ is said to be \emph{intermediate} if $\lim_{n\to \infty}k_n = \infty$ while $\lim_{n\to \infty} k_n/n=0$. It is well known that $F$ belongs to $\textsf{MDA}(\gamma)$  for some $\gamma>0$ if and only if, for all intermediate sequences $(k_n)$, $\widehat{\gamma}(k_n)$ converges in probability towards $\gamma$ \citep{mason1982, dehaanferreira2006}. Under mildly stronger conditions, it can be shown that $\sqrt{k_n} (\widehat{\gamma}(k_n)-\mathbb{E}\widehat{\gamma}(k_n))$ is asymptotically Gaussian with variance $\gamma^2.$ This suggests that, in order to minimise the quadratic risk $\mathbb{E}[(\widehat{\gamma}(k_n)-\gamma)^2]$  or the absolute risk $\mathbb{E}\left| \widehat{\gamma}(k_n)-\gamma\right|$, an appropriate choice for $k_n$ has to be made. If $k_n$ is too large, the Hill estimator $\widehat{\gamma}(k_n)$ suffers a large bias and, if $k_n$ is too small, $\widehat{\gamma}(k_n)$ suffers erratic fluctuations. 

As all estimators of the extreme value index face this dilemma \citep[see][and references therein]{BeiGoeTeuSeg04,dehaanferreira2006,Res07a}, during the last three decades, a variety of data-driven selection methods for $k_n$ has been proposed in the literature (see \citet{HaWe97}, \citet{hallwelsh1985}, \citet{MR1821820}, \citet{draismadehaanpengpereira1999}, \citet{MR1632189}, \citet{dreesdehaanresnick2000}, \citet{MR2435450}, \citet{carpentierkim2014} to name a few). 
A related but distinct problem is considered by 
\citet{carpentierkim2014b}: constructing uniform and adaptive confidence intervals for the extreme value index. 

The rationale for investigating adaptive Hill estimation stems from computational simplicity and variance optimality of properly chosen Hill estimators   \citep{beirlantbouquiauxwerker2006}. 

The hallmark of our approach is to combine techniques of \textsf{EVT} with tools from concentration of measure theory.
As up to our knowledge, the impact of the
concentration of measure phenomenon in \textsf{EVT} has received little attention, we comment and motivate the use of concentration arguments. Talagrand's concentration phenomenon for products of exponential distributions is one instance of a general phenomenon: concentration of measure in product spaces \citep{Led01, LeTa91}. The phenomenon may be summarised in a simple quote: functions of independent random variables that do not depend too much on any of them are almost constant \citep{Tal96}.  

The concentration approach helps to split the investigation in two steps: the first step consists in bounding the
fluctuations of the random variable under concern around its median or its expectation, while the second step focuses on the
expectation. This approach has seriously simplified the investigation of suprema of empirical processes and made the life
of many statisticians easier \citep{Tal96c, Tal05, massart:2003, Kol08}.  To point out the potential uses of
concentration inequalities in the field of \textsf{EVT} is one purpose of this paper. In statistics, concentration
inequalities have proved very useful when dealing with estimator selection and adaptivity issues: sharp, non-asymptotic tail bounds can be combined 
with simple union bounds in order to obtain uniform guarantees of the risk of collection of estimators. Using concentration
inequalities to investigate adaptive choice of the number of order statistics to be used in tail index estimation is a natural thing
to do. 

In the present setting, tail index estimators are functions of independent random variables. 
Talagrand's quote raises a first question: in which way are these tail functionals  smooth functions of independent random variables? We do not attempt here to revisit the asymptotic approach described by \citep{drees1998smooth} which equates smoothness with Hadamard differentiability. Our approach is non-asymptotic and our
conception of smoothness somewhat circular, smooth functionals are these functionals for which we can obtain good concentration inequalities. 

In this paper, we combine  Talagrand's concentration
inequality for smooth functions of independent exponentially
distributed random variables (Theorem \ref{bernstein:expo}) with three traditional  tools of \textsf{EVT}: the quantile transform, Karamata's representation for slowly  varying functions, and R\'enyi's characterisation of  the joint distribution of order statistics of exponential samples. This allows us to  establish concentration inequalities for the Hill process $(\sqrt{k}(\widehat{\gamma}(k)-\mathbb{E}\widehat{\gamma}(k))_{k})$ (Theorem \ref{prp:hill:concentration}
) in Section \ref{sec:conc-ineq-hill}. 

In Section \ref{sec:adapt-hill-estim}, we build on these concentration inequalities to analyse the performance of a variant of Lepki's rule defined in Sections \ref{sec:lepsk-meth-adapt} and   \ref{sec:adapt-hill-estim}: Theorem \ref{thm:adapt-hill-estim} describes an oracle inequality and Corollary \ref{cor:adapt-hill-estim-2ndrv} assesses the performance of this simple selection rule under a mild assumption on the so-called von Mises function.
Note that the  condition is less demanding than the regular variation condition on the von Mises function that has often been assumed when looking for adaptive tail index estimators (notable exceptions being \citep{carpentierkim2014} and \citep{MR2435450}).  It reveals that the performance of Hill estimators selected by Lepski's method matches known lower bounds (see Section \ref{sec:lower-bound}) that is, they suffer the loss of efficiency which is inherent to this problem, but not more.

Proofs 
are given in Section \ref{sec:proofs}. Finally,  in Section \ref{sec:simulations}, we examine the performance of this  adaptive Hill estimator for finite sample sizes using Monte-Carlo simulations.

\section{Background, notations and tools}
\label{sec:backgr-notat-tools}

\subsection{The Hill estimator as a smooth tail statistics}

The quantile function $F^\leftarrow$ is the generalised inverse of the distribution function $F$. The \emph{tail quantile function} of $F$ is a non-decreasing function defined on $(1,\infty)$ by
$U=(1/(1-F))^{\leftarrow}$, or by
\begin{displaymath}
U(t) = \inf \{  x~:~F(x)\geq 1-1/t  \}= F^\leftarrow(1-1/t)\pt 
\end{displaymath} 

In this text, we use a variation of the quantile transform that fits \textsf{EVT}: if $E$ is exponentially distributed, then $U(\exp (E))$ is distributed according to $F$. 
Moreover, by the same argument, 
the order statistics $X_{(1)} \geq \ldots \geq X_{(n)}$ are distributed as a monotone transformation of the order statistics $Y_{(1)} \geq \ldots \geq Y_{(n)}$ of a sample of $n$ independent standard exponential random variables. 
\begin{displaymath}
(X_{(1)},\ldots,X_{(n)}) \stackrel{d}{=} \left(U(\mathe^{Y_{(1)}}), \ldots, U(\mathe^{Y_{(n)}}) \right) \, .
\end{displaymath}

Thanks to R\'enyi's representation for order statistics of exponential samples, agreeing on $Y_{(n+1)}=0$,  the rescaled exponential spacings $Y_{(1)}-Y_{(2)}, \ldots,i(Y_{(i)}-Y_{(i+1)}), \ldots, (n-1)(Y_{(n-1)}-Y_{(n)}), n Y_{(n)}$  are independent and  exponentially distributed.

The quantile transform and R\'enyi's representation are complemented by Karamata's representation for slowly varying functions. Recall that a  function $L$ is \emph{slowly varying at infinity} if for all $x>0$, $\lim_{t\to \infty } L(tx)/L(t)= x^0= 1$.  The von Mises condition specifies the form of Karamata's representation 
\citep[see][Corollary 2.1]{Res07a} of the slowly varying component $t^{-\gamma}U(t)$ of $U(t)$. 

\begin{dfn}[\textsc{von Mises condition}] \label{dfn:vmises:cond} A distribution function $F$ belonging to $\textsf{MDA}(\gamma), \gamma>0$, satisfies the von Mises condition if  there exist a constant $t_0\geq 1$, a constant $c=U(t_0)t_0^{-\gamma}$ and a measurable function $\eta$ on
  $(1,\infty)$ such that, for $t\geq t_0$
  \begin{displaymath}
U(t)=  c t^\gamma  \exp\left(\int_{t_0}^t\frac{\eta(s)}{s} \mathrm{d}s   \right) 
\end{displaymath}
with $\lim_{s\to \infty}\eta(s)=0$.
The function $\eta$ is called the \emph{von Mises function}.
\end{dfn}
In the sequel, we assume that the sampling distribution $F\in \mathsf{MDA}(\gamma)$, $\gamma>0$, satisfies the von Mises condition with $t_0=1$, von Mises function $\eta$ and define the non-increasing function $\overline{\eta}$
from $[1,\infty)$ to $[0,\infty)$ by $\overline{\eta}(t)=\sup_{s\geq t} |\eta(s)|$.  In the text, we assume that $\overline{\eta}(1)<\infty$.

Combining the quantile transformation, R\'enyi's and Karamata's representations, it is straightforward that, under the von Mises condition,  the sequence  of Hill estimators is distributed as a  function of the largest order statistics of a standard exponential sample. 

\begin{prp}\label{hill:rep}
The vector of Hill estimators $(\widehat{\gamma}(k))_{k< n}$ is distributed as the random vector
\begin{equation}\label{eq:hill:rep}
\left( \frac{1}{k} \sum_{i=1}^k \int_{0}^{E_i} \left(\gamma +\eta(\mathe^{\frac{u}{i}+Y_{(i+1)}})\right) \mathrm{d}u   \right)_{k< n}
\end{equation}
where $E_1,\ldots, E_n$ are  independent standard exponential random variables while, for $i \leq n$, $Y_{(i)}=\sum_{j=i}^n E_j/j$ is distributed like the $i$th order statistic of an $n$-sample of the exponential distribution. 
\end{prp}
For a fixed $k<n$, a second distributional representation is available, 
\begin{equation}\label{rep:hill:var}
\widehat{\gamma}(k) \dis \frac{1}{k} \sum_{i=1}^k \int_{0}^{E_i} \left(\gamma +\eta(\mathe^{u+Y_{(k+1)}})\right) \mathrm{d}u  
\end{equation}
where $E_1,\ldots, E_k$ and $Y_{(k+1)}$ are defined as in Proposition \ref{hill:rep}.

This second, simpler, distributional representation stresses the fact that, conditionally on $Y_{(k+1)}$, $ \widehat{\gamma}(k) $ is distributed as a mixture of sums of independent random variables approximately distributed as exponential random variables with scale $\gamma$.  This distributional identity suggests that the variance of $\widehat{\gamma}(k)$ scales like $\gamma^2/k$, an intuition that is corroborated by analysis, see Section~\ref{sec:conc-ineq-hill}. 

The bias of $\widehat{\gamma}(k)$ is connected with the von
Mises function $\eta$ by the next formula
\begin{displaymath}
  \EXP \widehat{\gamma}(k) -\gamma = \EXP \left[ \int_0^\infty \mathe^{-v}
  \eta\left(\mathe^{Y_{(k+1)}} \mathe^v\right) \mathrm{d}v \right]=  \EXP \left[\int_1^\infty 
  \frac{\eta\left(\mathe^{Y_{(k+1)}} v\right)}{v^2} \mathrm{d}v\right]\, .
\end{displaymath}
     
Henceforth, let $b$ be defined on $(1,\infty)$ by 
\begin{equation}\label{cond:bias}
b(t)=\int_1^\infty  \frac{\eta\left(t v\right)}{v^2} \mathrm{d}v =t \int_t^\infty \frac{\eta\left( v\right)}{v^2} \mathrm{d}v \, .
\end{equation}
The quantity  $b(t)$ is the bias of the Hill estimator $\widehat{\gamma}(k)$ given $\overline{F}(X_{(k+1)})=1/t$. 
The second expression for $b$ shows that $b$ is differentiable with respect to $t$ (even though $\eta$ might be nowhere differentiable) and that 
\begin{displaymath}
b'(t)=\frac{b(t)-\eta(t)}{t} \, .
\end{displaymath}
The von Mises function governs both the rate of convergence of $U(tx)/U(t)$ towards $x^\gamma$, or equivalently of $\overline{F}(tx)/\overline{F}(t)$ towards $x^{-1/\gamma}$,
 and the rate of convergence of $|\mathbb{E} \widehat{\gamma}(k)-\gamma|$ towards $0$.

\subsection{Frameworks}
\label{sec:frameworks}

The difficulty in extreme value index estimation stems from the fact that, for any collection of estimators, for any intermediate sequence $(k_n)$, and for any $\gamma>0$, there is a distribution function $F \in \mathsf{MDA}(\gamma)$ such that the bias $|\mathbb{E} \widehat{\gamma}(k_n)-\gamma|$ decays at an arbitrarily slow rate. This has led authors to put conditions on the rate of convergence of $U(tx)/U(t)$  towards $x^\gamma$ as $t$ tends to infinity while $x>0$, or equivalently, on the rate of convergence of $\overline{F}(tx)/\overline{F}(t)$  towards $x^{-1/\gamma}$. These conditions have then to be translated into conditions on  the rate of decay of the bias of estimators. As we focus on Hill estimators, the connection between the rate of convergence  of $U(tx)/U(t)$  towards $x^\gamma$
and the rate of decay of the bias is transparent and well-understood
\citep{MR1925570}: the theory of  $O$-regular variation provides an adequate setting for describing both rates of convergence \citep{BiGoTe87}. 
In words, if a positive  function  $g$ defined over $[1,\infty)$ is such that, for some
$\alpha\in \mathbb{R}$, for all $\Lambda>1$, $\limsup_t \sup_{x \in [1,\Lambda]}
g(tx)/g(t) < \infty$, $g$ is said to have \emph{bounded increase}. If $g$ has bounded increase, the class
$O\Pi_g$ is the class of measurable    functions $f$ on some interval 
$[a,\infty), a>0$, such that as $t\to \infty,$ $f(tx)-f(t)= O(g(t))$  for all $x\geq 1.$

For example, the analysis carried out by \citet{carpentierkim2014} rests on  the condition that, if $F \in \mathsf{MDA}(\gamma)$, 
 for some  $C>0$, $D \neq 0$ and $\rho<0$,
\begin{equation}\label{eq:kimcarpcond}
  \left| \frac{\overline{F}(x)}{x^{-1/\gamma}} - C \right| \leq Dx^{\rho/\gamma} \, .
\end{equation}
This condition  implies that $\ln(t^{-\gamma}U(t))\in O\Pi_g$ with $g(t)=t^\rho$ \citep[p. 473]{MR1925570}.
Thus, under the von Mises condition, Condition \eqref{eq:kimcarpcond}  implies that
the function $\int_t^\infty (\eta(s)/s) \mathrm{d}s $ belongs to $O\Pi_g$
with $g(t)=t^\rho.$ 
Moreover,  the Abelian and Tauberian theorems from \citep{MR1925570} assert that 
$ \int_t^\infty (\eta(s)/s) \mathrm{d}s \in O\Pi_g$ if and only if 
$|\mathbb{E} \widehat{\gamma}(k_n) -\gamma|=O(g(n/k_n))$ for any
intermediate sequence $(k_n)$. 

In this text, we are ready to assume that if $F \in \mathsf{MDA}(\gamma)$ and satisfies the von Mises condition, then,  for some $C>0$ and $\rho<0$ and $t>1$, 
\begin{displaymath}
	|\overline{\eta}(t)| \leq C t ^\rho \, . 
\end{displaymath}
This condition is arguably more stringent than (\ref{eq:kimcarpcond}). 
However, we  do not want to assume that $\eta$ 
satisfies a  regular variation property. This would imply that $t\mapsto |b(t)|$ is $\rho$-regularly varying.

Indeed, assuming as in \citep{hallwelsh1985} and several subsequent papers that 
$F$ satisfies
\begin{equation}\label{hall:condition}
\overline{F}(x)=Cx^{-1/\gamma} \left(1+Dx^{\rho/\gamma} + o(x^{\rho/\gamma})\right)
\end{equation}
where $C>0, D\neq0$ are constants and $\rho <0$, or equivalently,\citep*{CsoDeMa85,MR1632189}
that $U$ satisfies  
\begin{displaymath}
U(t) {=}C^{\gamma}t^{\gamma} \left( 1 + \gamma
  DC^{\rho} t^{\rho} + o(t^{\rho}) \right) \, 
\end{displaymath}
(which entails that $\eta$ is regularly varying)
makes the problem of extreme value index estimation easier (but not easy).
 These conditions entail that, for any intermediate sequence $(k_n)$, the ratio $|\mathbb{E}[\widehat{\gamma}(k_n)-\gamma]|/(n/k_n)^\rho$
converges towards a finite limit  as $n$ tends to $\infty$
\citep{BeiGoeTeuSeg04,dehaanferreira2006,MR1925570}. This makes the estimation of the
second-order parameter a very natural intermediate objective \citep[see for example][]{MR1632189}. 

\subsection{Lepski's method and adaptive tail index estimation}
\label{sec:lepsk-meth-adapt}
The necessity of developing data-driven index selection methods is illustrated in Figure \ref{fig:riskcomp-student}, which displays the estimated standardised  
root mean squared error (\textsc{rmse}) of Hill estimators 
$$\mathbb{E}\left[\left(\frac{\widehat{\gamma}(k)}{\gamma} -1\right)^2\right]^{1/2}$$  
as a function of $k$ for four related sampling distributions which all
satisfy the second-order  condition \eqref{hall:condition} with
different values of the second-order parameters.
\begin{figure}[htb]
 \centering
  \includegraphics[width=.9\textwidth]{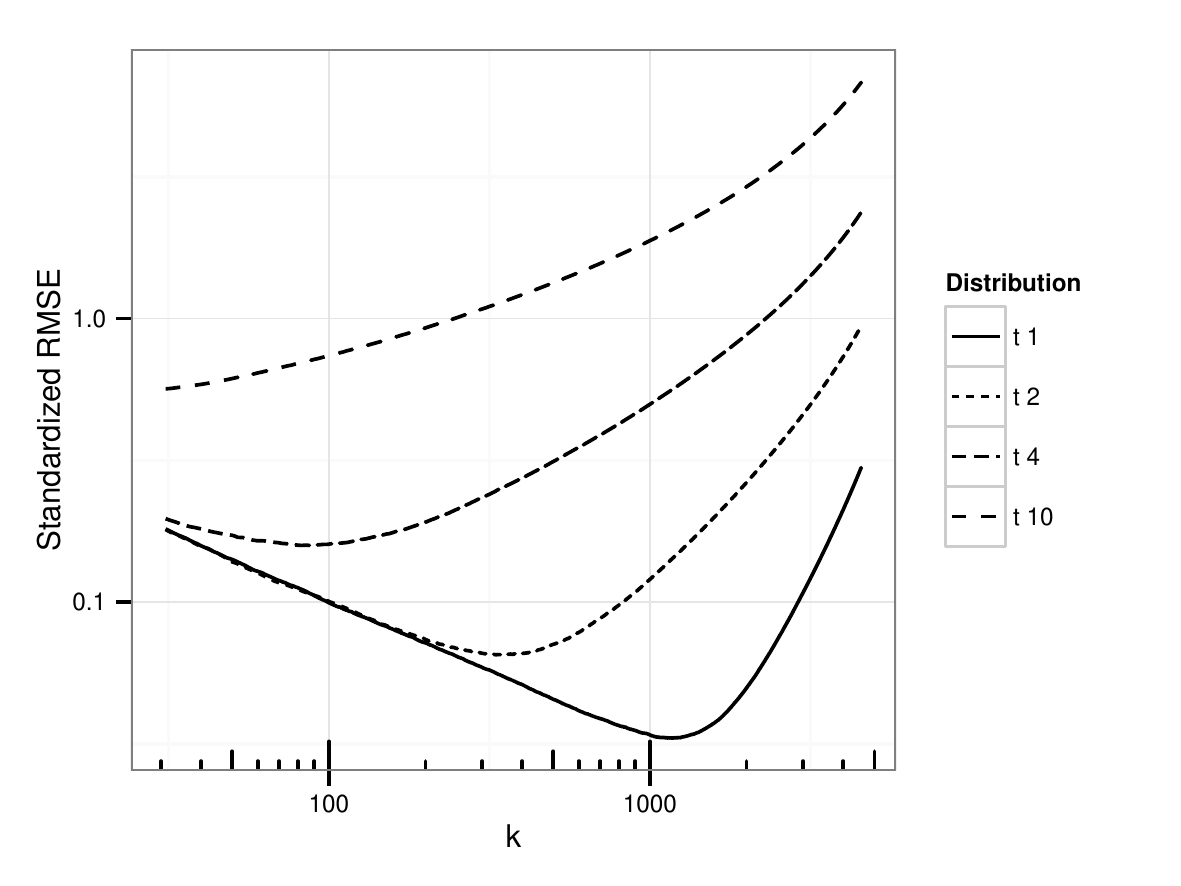}
    \caption{Estimated standardised \textsc{rmse}
      as a function of $k$ for samples of size $10000$ from Student's
      distributions with different degrees of freedom
      $\nu=1,2,4,10$. All four distributions satisfy Condition
      \eqref{hall:condition} with $|\rho|=2/\nu$. The increasing parts
      of the lines reflect the values of $\rho$. \textsc{rmse} is
      estimated by averaging over $5000$ Monte-Carlo simulations.}
    \label{fig:riskcomp-student}
\end{figure}

Under this second-order condition \eqref{hall:condition}, \citeauthor{hallwelsh1985} proved that the asymptotic mean squared error of the Hill estimator is minimal for 
sequences $(k_n^*)_n$ satisfying
\begin{displaymath}
k^*_n \sim K(C,D,\rho) \, n^{2|\rho|/(1+2|\rho|)} 
\end{displaymath}
with $K(C,D,\rho)=  \left({C^{2|\rho|}(1+|\rho|)^2}/{2D^2|\rho|^3}\right)^{1/(1+2|\rho|)}$.
Since $C>0$, $D \neq 0$ and the second-order parameter $\rho<0$ are
usually  unknown, many authors have been interested in the
construction of data-driven selection procedures for $k_n$ under
conditions such as \eqref{hall:condition}. A great deal of ingenuity has been dedicated to the estimation of the second-order parameters and to 
the use of such estimates when estimating first order parameters. 

 As we do not want to assume a second-order condition such as Condition \eqref{hall:condition}, we resort to Lepski's method which is a general attempt to balance bias and variance. 

Since its introduction \citep{lepski1991}, this  general method for model selection  has been
proved to achieve adaptivity and to provide one with oracle inequalities  in a variety of inferential contexts ranging
from density estimation to inverse problems and  classification \citep{lespkitsybakov2000,lepski1991,lepski1990,lepski1992}. 
Very readable introductions to Lepski's method and its connections
with penalised contrast methods can be found in \citep{birge:2001,mathe2006}.
In \textsf{EVT}, we are aware of three papers that explicitly
rely on this methodology:  \citep{MR1632189},
\citep{MR2435450} and \citep{carpentierkim2014}.  

The selection rule analysed in the present paper (see Section \ref{sec:adapt-hill-estim} for a precise definition) is a variant of  the preliminary selection rule introduced in \citep{MR1632189}
\begin{equation}\label{eq:rule:dk}
  \overline{\kappa}_n(r_n) = \min \left\{ k \in \{ 2,\ldots,n\} \colon \max_{2\leq i\leq k} \sqrt{i}|\widehat{\gamma}(i)-\widehat{\gamma}(k)|> r_n  \right\}
\end{equation}
where $(r_n)_n$ is a sequence of thresholds such that $\sqrt{\ln \ln n} =  o(r_n)$  and $r_n = o(\sqrt{n})$, and $\widehat{\gamma}(i)$ is the Hill estimator computed from the $(i+1)$ largest order statistics. The definition of this ``stopping time" is motivated  by Lemma 1 from \citep{MR1632189} which asserts that, under the von Mises condition, 
\begin{displaymath}
\max_{2 \le i \le k_n} \sqrt{i} |\widehat{\gamma}(i)-\esp\left[ \widehat{\gamma}(i)\right] | = O_P \left( \sqrt{\ln \ln n} \right) \pt
\end{displaymath}
In words, this selection rule almost picks out the largest index $k$ such
that, for all $i$ smaller than $k$,  $\widehat{\gamma}(k)$ differs from $\widehat{\gamma}(i)$ by a
quantity that is not much larger than the typical fluctuations of
$\widehat{\gamma}(i)$. This index selection rule can be performed graphically by interpreting an alternative Hill plot as shown on Figure \ref{fig:cauchy-lepski}
\citep[see][for a discussion on the merits of alt-Hill plots]{dreesdehaanresnick2000,Res07a}.

\begin{figure}[thb]\centering
 \includegraphics[width=.9\textwidth]{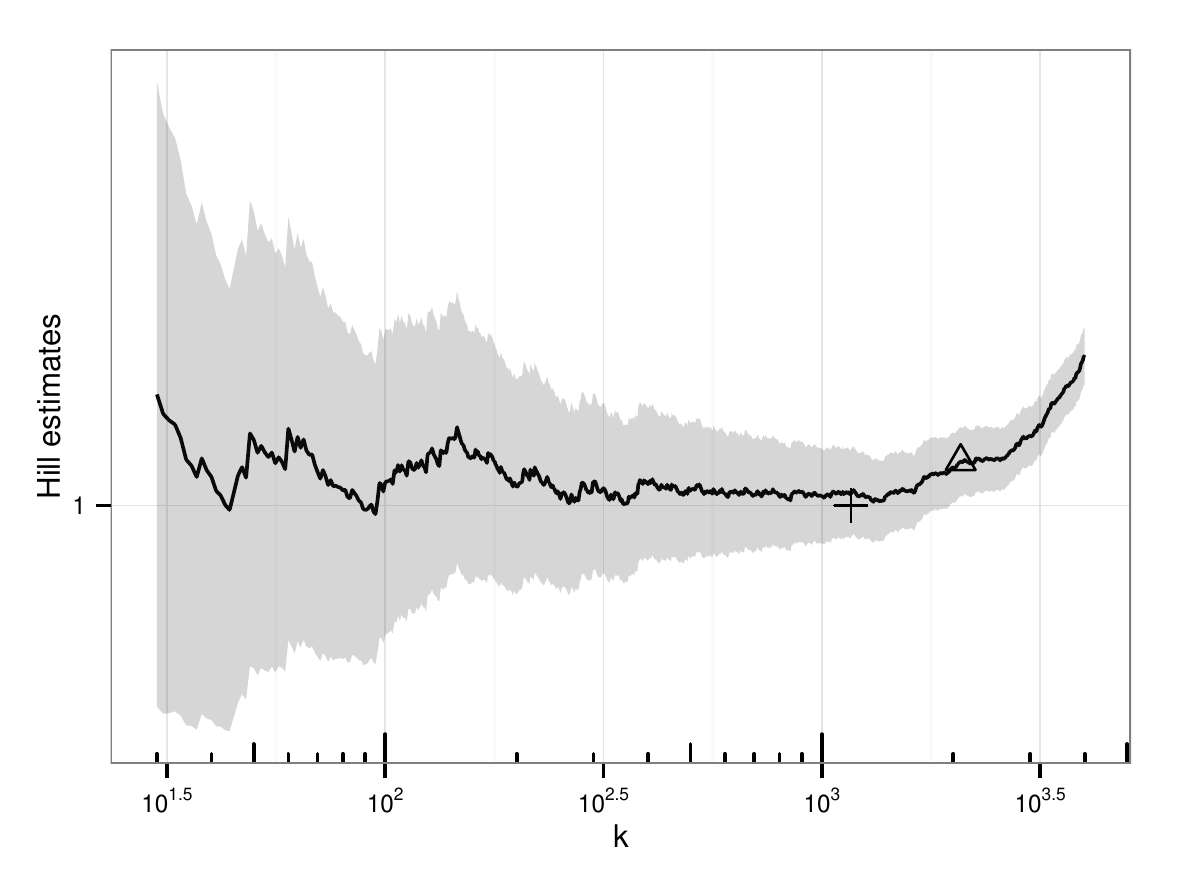}
  \caption{Lepski's method illustrated on a alt-Hill plot. The plain line describes the sequence of Hill estimates as a function of index $k$ computed on a pseudo-random sample of size $n = 10 000$ from Student distribution with 1 degree of freedom (Cauchy distribution). Hill estimators are computed from the positive order statistics. The grey ribbon around the plain line provides a graphic illustration of Lepski's method. For a given value of $i$, the width of the ribbon is $2 r_n \widehat{\gamma}(i)/\sqrt{i}$. A point $(k,\widehat{\gamma}(k))$ on the plain line corresponds to an eligible index if the horizontal segment between this point and the vertical axis lies inside the ribbon that is, if for all $i, 30\leq i< k$, $|\widehat{\gamma}(k)-\widehat{\gamma}(i)|\leq r_n \widehat{\gamma}(i)/\sqrt{i}$. If $r_n$  were replaced by an appropriate quantile of the Gaussian distribution, the grey ribbon would just represent the confidence tube that is usually added on Hill plots. The triangle represents the selected index with $r_n = \sqrt{2.1 \ln \ln n}$. The cross represents the oracle index estimated from Monte-Carlo simulations, see Table \ref{tab:ratios:1}.}
  \label{fig:cauchy-lepski}
\end{figure}

The goal of \citet{MR1632189}  is not to investigate the performance of the preliminary selection rule defined in Display \eqref{eq:rule:dk} but to design a selection rule $\widehat{\kappa}_n(r_n)$, based on $\overline{\kappa}_n(r_n)$, that would asymptotically mimic the optimal selection rule $k^*_n$ under second-order conditions. 

Our goal, as in \citep{MR2435450,carpentierkim2014}, is to derive
non-asymptotic risk bounds without making a second-order assumption. 
 In both papers,  the rationale for working with some
special  collection of estimators seems to be the ability to derive
non-asymptotic deviation  inequalities for $\widehat{\gamma}(k)$ either
from  exponential inequalities for log-likelihood ratio statistics or from  simple binomial tail inequalities such as Bernstein's inequality \citep[see][Section 2.8]{BoLuMa13}. 

In models satisfying Condition \eqref{hall:condition}, the
estimators from  \citep{MR2435450}  achieve the optimal rate up to a
$\ln(n)$ factor.
\citet{carpentierkim2014} prove that the risk  of their data-driven
estimator decays at the optimal rate $n^{|\rho|/(1+2|\rho|)}$ up to a
factor $r_n^{2|\rho|/(1+2|\rho|)}= (\ln\ln n)^{|\rho|/(1+2|\rho|)}$ in
models satisfying Condition \eqref{eq:kimcarpcond}. 

We aim at achieving optimal risk bounds under Condition
\eqref{eq:kimcarpcond} using a simple estimation method requiring
almost no calibration effort and based on mainstream extreme value
index estimators. Before describing the keystone of our approach in
Section \ref{sec:talagr-conc-phen}, we recall the recent lower risk
bound for adaptive extreme value index estimation.

\subsection{Lower bound}
\label{sec:lower-bound}
One of the key results in \citep{carpentierkim2014} is a lower bound on the accuracy of adaptive tail index estimation. 
This lower bound reveals that, just as for estimating a density at a
point \citep{lepski1991,lepski1992}, or point estimation in Sobolev
spaces \citep{MR1700239}, as far as tail index estimation is concerned, adaptivity has a price. Using Fano's Lemma, and a Bayesian game that extends cleanly in frameworks of \citep{MR2435450} and \citep{Nov14}, \citeauthor{carpentierkim2014} were able to prove the next minimax lower bound.

\begin{thm}\label{thm-kc-lower-bound}
Let $\rho_0 <-1$ and $ v \in [0, \mathe/(1+2\mathe)]$. Then, for any tail index estimator $\widehat{\gamma}$ and any sample size $n$ such that $M=\lfloor \ln n\rfloor> \mathe/v$, there exists a probability distribution $P$ such that
\begin{enumerate}[i)]
\item $P \in \mathsf{MDA}(\gamma)$  with $\gamma>0$,
\item $P$ meets the von Mises condition with von Mises function $\eta$ satisfying
\begin{displaymath}
  \overline{\eta}(t) \leq \gamma t^{\rho}
\end{displaymath}
for some $\rho \in [\rho_0, 0)$,
\item \begin{displaymath}
  P\left\{
    |\widehat{\gamma}-\gamma|\geq \frac{C_\rho}{4} \gamma
    \left(\frac{v\ln\ln n}{n}\right)^{|\rho|/(1+2|\rho|)} \right\} \geq \frac{1}{1+2\mathe} 
\end{displaymath}
and 
\begin{displaymath}
 \EXP_P \left[ \frac{|\widehat{\gamma}-\gamma|}{\gamma} \right] \geq \frac{C_\rho}{4(1+2\mathe)}  \left(\frac{v\ln\ln n}{n}\right)^{|\rho|/(1+2|\rho|)} \, , 
\end{displaymath}
with $C_\rho=1-\exp\left(-\tfrac{1}{2(1+2|\rho|)^2}\right)$. 
\end{enumerate}
\end{thm}

Using Birg\'e's Lemma instead of Fano's Lemma, we provide a simpler, shorter proof of this theorem (see Appendix \ref{proof:lower:bound}). 

The lower rate of convergence provided by Theorem
\ref{thm-kc-lower-bound} is another incentive to revisit the
preliminary tail index estimator from \citep{MR1632189}. However, instead of
using a sequence $(r_n)_n$ of order larger than $\sqrt{\ln \ln n}$ in
order to calibrate pairwise tests and ultimately to design estimators
of the second-order parameter (if there are any), it is worth
investigating a minimal sequence where $r_n$  is of order $\sqrt{\ln
  \ln n}$, and check whether the corresponding adaptive estimator
achieves the Carpentier-Kim lower bound (Theorem
\ref{thm-kc-lower-bound}).

In this paper, we focus on $r_n$ of the order $\sqrt{\ln \ln n}$. The rationale for imposing  $r_n$ of the order $\sqrt{\ln \ln n}$ can be understood by the fact that, even if the sampling distribution is a pure Pareto distribution with shape parameter $\gamma$  ($\overline{F}(x)= (x/\tau)^{-1/\gamma}$  for $x\geq \tau>0$), if 
$$\limsup  r_n/(\gamma \sqrt{2 \ln \ln n})<1 \, ,$$ the preliminary selection rule
 will, with high
probability, select a small value of $k$ and thus pick out a suboptimal
estimator.  This can be justified  using results  from \citep{DarErd56} (see Appendix \ref{sec:calibr-prel-select} for details). 

Such an endeavour requires sharp probabilistic tools. They are the topic of the next section.

\subsection{Talagrand's concentration phenomenon for products of exponential distributions}
\label{sec:talagr-conc-phen}

Deriving authentic concentration inequalities for Hill estimators is not straightforward. Fortunately, the construction of such inequalities turns out to be possible thanks to general functional inequalities that hold for functions of independent exponentially distributed random variables. We recall these inequalities (Proposition \ref{poincare:expo} and Theorem \ref{bernstein:expo}) which have been largely overlooked in statistics. A thorough and readable presentation of these inequalities can be found in \citep{Led01}. We start by the easiest result, a variance bound  that pertains to the family of Poincar\'e inequalities.

\begin{prp}[Poincar\'e inequality for exponentials, \citep{BoLe97}]
\label{poincare:expo}
 If $g$ is a differentiable function over $\mathbb{R}^n$ and $Z=g(E_1, \ldots,E_n)$  where $E_1,\ldots,E_n$ are independent standard exponential random variables, then 
  \begin{displaymath}
    \var(Z)\leq 4 \EXP \left[ \left\| \nabla g \right\|^2\right] \, . 
  \end{displaymath}
\end{prp}
\begin{rem}
The constant $4$  can not be improved. 
\end{rem}

The next corollary is stated in order to point the relevance of this
Poincar\'e  inequality to the analysis of general order statistics and
their functionals. 
Recall that the \emph{hazard rate} of an absolutely continuous probability distribution with distribution $F$ is: $h=f/\overline{F}$ where $f$ and $\overline{F}=1-F$ are the density and the survival function associated with $F$, respectively.

\begin{cor} \label{var:os}
Assume the distribution of $X$ has a positive density, then the $k$th order statistic $X_{(k)}$ satisfies
   \begin{displaymath}
      \var(X_{(k)})  \leq C \sum_{i=k}^n \frac{1}{i^2}\EXP \left[ \frac{1}{h(X_{(k)})^2} \right] \le \frac{C}{k} \left(1 +\frac{1}{k}\right)\EXP \left[ \frac{1}{h(X_{(k)})^2} \right]  
    \end{displaymath}
where $C$ can be chosen as $4$.
\end{cor}
\begin{rem}
By Smirnov's Lemma \citep{dehaanferreira2006}, $C$ can not be smaller than $1$. If the distribution of $X$ has a non-decreasing hazard rate, the factor of $4$ can be improved into a factor $2$ \citep{boucheronthomas2012}. 
\end{rem}

\citet{Tal91, Mau91, BoLe97} show that smooth functions of independent exponential random variables satisfy Bernstein type concentration inequalities. The next result is extracted from the derivation of Talagrand's concentration phenomenon for product of exponential random variables in \citep{BoLe97}.
  
The definition of sub-gamma random variables will be used in the formulation of the theorem and in many arguments. 

\begin{dfn}\label{dfn:sub-gamma}
A real-valued centred random variable $X$ is said to be \emph{sub-gamma} on the right tail with variance factor $v$ and scale parameter $c$ if
\begin{displaymath}
  \ln \EXP \mathe^{\lambda X} \leq\frac{\lambda^2v}{2(1-c\lambda) } \text{ for every }\lambda\quad \mbox{such that} \quad
  0<\lambda<1/c \, .
\end{displaymath}
We denote the collection of such random variables by $\Gamma_+(v,c)$. Similarly, $X$ is said to be {sub-gamma on the
left tail with variance factor $v$ and scale parameter $c$} if $-X$ is sub-gamma on the right tail with variance factor $v$ and
tail parameter $c$. We denote the collection of such random variables by $\Gamma_{-}(v,c)$ and $\Gamma_{+}(v,c)\cap \Gamma_-(v,c)$ by $\Gamma_\pm(v,c)$.
\end{dfn}

If $X -\EXP X\in  \Gamma_+(v,c)$, then for all $\delta\in (0,1)$,  with probability larger than $1-\delta,$
\begin{displaymath}
X \leq \EXP X + \sqrt{2v \ln \left(1/\delta \right)} + c \ln \left(1/\delta \right) \, . 
\end{displaymath}

The entropy of a non-negative random variable $X$ is defined by $\ent [X]= \esp [X \ln X]-\esp X \ln\esp X$. 

\begin{thm}
\label{bernstein:expo}
 Assume that $g$ is a differentiable function on $\mathbb{R}^n$  with $\max_i |\partial_i g| < \infty$. Let $Z=g(E_1, \ldots,E_n)$ where $E_1,\ldots, E_n$ are $n$ independent standard exponential random variables and $c<1$. Then, for all $\lambda$ such that $0\leq \lambda \max_i |\partial_i g|  \leq c$,
\begin{displaymath}
\ent \left[ \mathe^{\lambda (Z-\EXP Z)}\right]  \leq \frac{2\lambda^2}{1-c} \EXP \left[ \mathe^{\lambda (Z-\EXP Z)} \| \nabla g\|^2\right] \, .
\end{displaymath}

Let $v$ be the essential supremum of $\| \nabla g\|^2$, then $Z$ is sub-gamma on both tails with variance factor $4v$ and scale factor $\max_i |\partial_i g|$.
\end{thm}

Again, we illustrate the relevance of these versatile tools on the analysis of general order statistics. This general theorem 
implies that  if the sampling distribution has non-decreasing hazard rate, then the order statistics $X_{(k)}$ satisfy  Bernstein type inequalities \citep[see][Section 2.8]{BoLuMa13} with variance factor $4/k\EXP \left[ {1}/{h(X_{(k)})^2} \right]$ (the Poincar\'e estimate of  variance) and scale parameter $(\sup_x 1/h(x))/k$). Starting back from the Efron-Stein-Steele inequality, the authors derived a somewhat sharper inequality \citep{boucheronthomas2012}.  

{\sloppy
\begin{cor}\label{prp:hazard:conc:ineg}
Assume the distribution function $F$ has non-decreasing hazard rate $h$ that is, $U \circ\ \exp$ is $C^1$ and concave. Let  $Z= g(E_1,\ldots, E_n)=\left(U\circ\ \exp\right)\left( \sum_{i=k}^n E_i/i\right)$ be distributed as the $k$th order statistic
of a sample distributed according to $F$. Then, $Z$ is sub-gamma on both tails with variance factor $4/k\left(1+1/k\right) \EXP [1/h(Z)^2]$
and scale factor $1/(k \inf_x h(x))$. 
\end{cor}}

This corollary describes in which way central, intermediate and extreme order statistics can be portrayed as smooth functions of independent exponential random variables. This possibility should not be taken for granted as it is non trivial to capture in  a non-asymptotic way the tail behaviour of maxima of independent Gaussians \citep{Led01,boucheronthomas2012,Cha14}. In the next section, we show in which way the Hill estimator can fit into this picture. 

\section{Main results}

In this section, the sampling distribution $F$ is assumed to belong to $\textsf{MDA}(\gamma)$ with  $\gamma>0$ and to  satisfy the von Mises condition (Definition \ref{dfn:vmises:cond}) with bounded von Mises function $\eta$. 

\subsection{Variance and concentration inequalities for the Hill estimators}
 \label{sec:conc-ineq-hill}
 It is well known that, under the von Mises condition, if $(k_n)$   is an intermediate sequence, the sequence 
$\sqrt{k_n} \left(\widehat{\gamma}(k_n)- \EXP \widehat{\gamma}(k_n) \right)$ converges in distribution
towards $\mathcal{N}(0,\gamma^2)$,  suggesting that the variance of $\widehat{\gamma}(k_n)$ scales like 
$\gamma^2/k_n$  \citep[see][]{gelukdehaanresnickstarica1997,BeiGoeTeuSeg04,dehaanferreira2006,Res07a}.


Proposition \ref{prop:bound-vari-hill} provides us with handy non-asymptotic bounds 
on $ \var[ \widehat{\gamma}(k)]  - \gamma^2/k $ using the von Mises function.
 
 \begin{prp}\label{prop:bound-vari-hill}
 Let $\widehat{\gamma}(k)$ be the Hill estimator computed from 
the $(k+1)$ largest order statistics of an $n$-sample from $F$. Then, 
    \begin{displaymath}
      -
    \frac{2\gamma}{k} \EXP \left[
      \overline{\eta}\left(\mathe^{Y_{(k+1)}} \right)\right] \leq 
\var[ \widehat{\gamma}(k)]  - \frac{\gamma^2}{k} 
\leq 
 \frac{2\gamma}{k} \EXP \left[ \overline{\eta}\left(\mathe^{Y_{(k+1)}} \right)\right] 
+ \frac{5}{k}\EXP\left[ \overline{\eta}\left(\mathe^{Y_{(k+1)}} \right)^2\right]  \, .
    \end{displaymath}
\end{prp}

The next  Abelian result might help in appreciating these variance
bounds. 
\begin{prp}
  \label{sec:abel-vari-hill}
Assuming that $\eta$ is $\rho$-regularly varying with $\rho<0$, then, 
for any intermediate sequence $(k_n)$,
\begin{displaymath}
\lim_{n\to \infty} \frac{  k_n \var(\widehat{\gamma}(k_n)) -\gamma^2 }{\eta(n/k_n)} = \frac{2\gamma}{(1-\rho)^2} \, . 
\end{displaymath}
\end{prp}
We may now move to genuine concentration inequalities for the Hill
estimator. 


The exponential representation
\eqref{eq:hill:rep} suggests that the rescaled Hill estimator $k \widehat{\gamma}(k)$ should be approximately distributed according to a $\text{gamma}(k,\gamma)$ distribution where $k$ is the shape parameter and $\gamma$ the scale parameter. Therefore, we expect the Hill estimators to satisfy Bernstein type concentration inequalities that is, to be sub-gamma on both tails with variance factors connected to the tail index $\gamma$ and to the von Mises function. Representation \eqref{eq:hill:rep} actually suggests more. Following \citep{MR1632189}, we actually expect the sequence $\big(\sqrt{k}(\widehat{\gamma}(k)-\EXP \widehat{\gamma}(k))\big)_k$ to behave like normalized partial sums of independent square integrable random variables that is, we believe $\max_{2\leq k\leq n} \sqrt{k}(\widehat{\gamma}(k)-\EXP \widehat{\gamma}(k))$ to scale like $\sqrt{\ln \ln n}$ and to be sub-gamma on both tails (see Appendix \ref{sec:calibr-prel-select}). The purpose of this section is to meet these expectations in a non-asymptotic way.  

Proofs use the Markov property of order statistics: conditionally on the $(J+1)$th order statistic, the first largest
$J$ order statistics are distributed as the order statistics of a sample of size $J$  of the excess distribution. They consist of appropriate invocations of Talagrand's concentration inequality (Theorem \ref{bernstein:expo}). However, this theorem
generally requires a uniform bound on the gradient of the relevant
function. When Hill estimators are analysed   as functions of
independent exponential random  variables, the partial derivatives
depend on the points at which the von Mises function is evaluated. In
order to get interesting bounds, it is worth conditioning on an intermediate order statistic. 

Throughout this subsection, let $\ell$ be an integer larger than $\sqrt{\ln  n}$ and $J$ an integer not larger  than $n$. We denote $E_i, 1 \le i \le n$, $n$ independent standard exponential random variables and we work on the probability space where all $E_i$ are defined, and therefore consider the Hill estimators defined by Representation \eqref{eq:hill:rep}.
  As we use the exponential representation of order statistics, besides Hill estimators, the random variables that appear in the main statements are order statistics of exponential samples. As before,  $Y_{(k)}$ will denote the $k$th order statistic of a standard exponential sample of size $n$ (we agree on $Y_{(n+1)}=0$).


The first theorem provides an exponential refinement of the variance
bound stated in  Proposition~\ref{prop:bound-vari-hill}.  
However, as announced, there is a price to pay:
statements hold conditionally on some order statistic. This is not an impediment to analyse Lepski's rule using this theorem. Indeed, when analysing Lepki's rule it is sufficient to control the Hill process $\left( \sqrt{i} (\widehat{\gamma}(i) - \mathbb{E}[\widehat{\gamma}(i) \mid Y_{(k+1)}])\right)_i$ for indices $i$ ranging between  $\ell_n$ (that should not be smaller than $\ln(n)$) and  some upper bound $k_n$  that achieves a certain balance between bias and standard deviation (the bias of $\widehat{\gamma}(k_n)$ should be of order $r_n$  times the standard deviation, that is approximately $\gamma/\sqrt{k_n}$ where $r_n \approx \sqrt{\ln(\ln(n))}$). The second clause of next theorem is the cornerstone in  the derivation of the risk bounds presented in the next section. 

In the sequel, let 
$$\xi_n= c_1\sqrt{\ln \log_2 n} + c_1',$$
where $c_1$ may be chosen not larger than $4$ and $c_1'$ not larger than $34$.
 
\begin{thm}\label{prp:hill:concentration}
Let $T$ be a shorthand for $\exp(Y_{(J+1)})$. For some $k$ such that $c_2\ln n \vee 32\leq \ell \le k\le J$ where $c_2\geq 2$, let
\begin{displaymath}
Z^a = \max_{\ell \leq i\leq k} \sqrt{i} \left| \widehat{\gamma}(i) -\EXP \left[\widehat{\gamma}(i)\mid Y_{(k+1)}\right] \right|\, .
\end{displaymath} 
Then, conditionally on $T$,
\begin{enumerate}[i)]
\item  For $\ell \leq i\leq k$, 
 $$\sqrt{i} \left( \widehat{\gamma}(i) -\EXP[\widehat{\gamma}(i)\mid T]\right) \in  \Gamma_\pm\left(4  \left(\gamma +3\overline{\eta}(T) \right)^2,\left(\gamma+2\overline{\eta}(T)\right)\right) \, . $$ 
\item  Let $u$ be such that $J \overline{\eta}(u)^2 \leq  \gamma^2 r_n^2$   where $r_n=\sqrt{c_3 \ln \ln n }$   with $c_3 = 2$. Assume that $T \geq u$, then 
$$Z^a \in \Gamma_\pm\left(4 \gamma^2(1 + 3 r_n/\sqrt{J})^2,\gamma(1 + 2r_n/\sqrt{J})/ \sqrt{\ell}\right) $$ and 
\begin{equation*}
    \EXP \left[ Z^a \mid T\right] \leq \gamma \xi_n \left(1 + \frac{3r_n}{\sqrt{J}}  \right)\, .
  \end{equation*}
\end{enumerate}
\end{thm}

\begin{rem}
If $F$  is a pure Pareto distribution with shape parameter $\gamma>0$, then $k\widehat{\gamma}(k)/\gamma$ is distributed 
according to a gamma distribution with shape parameter $k$ and scale parameter $1$. Tight and well-known tail bounds for gamma distributed random variables assert that
\begin{displaymath}
    \mathbb{P}\left\{ \left| \widehat{\gamma}(k)-\EXP\left[
          \widehat{\gamma}(k)\right] \right| \geq
      \frac{\gamma}{\sqrt{k}} \left( \sqrt{2\ln\left(2/\delta \right)}
        +\frac{\ln \left(2/\delta \right)}{\sqrt{k}}\right) \right\} \leq 2
    \delta \, . 
\end{displaymath}
\end{rem}

\begin{rem}
First part of Statement ii)  reads as: conditionally on  $\overline{\eta}(T) \leq r_n \gamma/\sqrt{J}$, with probability larger than $1-\delta$,
  \begin{displaymath}\label{eq:8}
    \left|Z^a-\EXP[Z^a  \mid T]\right| \leq  \gamma (1+3r_n /\sqrt{J}) \left( \sqrt{8 \ln \left(2/\delta \right) } +  \frac{\ln \left(2/\delta \right)}{\sqrt{\ell}}\right) \, . 
    \, 
  \end{displaymath}
     Combining both parts of Statement ii), we also get that, conditionally on  $\overline{\eta}(T) \leq r_n \gamma/\sqrt{J}$, with probability larger than $1-\delta$, 
\begin{displaymath}\label{eq:6:bis}
  Z^a \leq \gamma(1 + 3 r_n/\sqrt{J}) \left( \xi_n+\sqrt{8 \ln \left(2/\delta \right) } +  \frac{\ln \left(2/\delta \right)}{\sqrt{\ell}}\right)  
  \, . 
\end{displaymath}
\end{rem}

\begin{rem}
 The reader may wonder whether resorting to the exponential
 representation and usual Chernoff bounding would not provide a
 simpler argument. The straightforward approach leads to the 
 following conditional bound on the logarithmic moment generating function,
 \begin{eqnarray*}
\lefteqn{   \ln \EXP \left[ \exp\left(\lambda\left( \widehat{\gamma}(k) -\EXP
         [\widehat{\gamma}(k)\mid Y_{(k+1)}] \right) \right) \mid
     Y_{(k+1)}\right]}\\ &  \leq & \frac{\left(\gamma
       +\overline{\eta}(\mathe^{Y_{(k+1)}})\right)^2}{2k\left(1-\lambda(\gamma
       +\overline{\eta}(\mathe^{Y_{(k+1)}})) \right)} +\lambda
   \left(\overline{\eta}(\mathe^{Y_{(k+1)}})-b\left(
       \mathe^{Y_{(k+1)}}\right) \right) \, .
 \end{eqnarray*}
A similar statement holds for the lower tail. 
This leads to exponential bounds for deviations of the Hill
estimator above $\EXP
         [\widehat{\gamma}(k)\mid Y_{(k+1)}] + \overline{\eta}(\mathe^{Y_{(k+1)}})-b\left(
       \mathe^{Y_{(k+1)}}\right)$ that is, to control deviations of
     the Hill estimator above  its expectation plus a term that may be of the
     order of magnitude of the bias. 

Attempts to rewrite $\widehat{\gamma}(k)-\EXP [\widehat{\gamma}(k) \mid
Y_{(k+1)}]$ as a sum of martingale increments
$\EXP[\widehat{\gamma}(k)\mid Y_{(i)}]-\EXP [\widehat{\gamma}(k) \mid
Y_{(i+1)}]$, for $1\leq i\leq k$, and to exhibit an exponential
supermartingale met the same impediments.  

At the expense of inflating the
     variance factor, Theorem \ref{bernstein:expo} provides a genuine
    (conditional)   concentration inequality  for Hill estimators. As
    we will deal with values of $k$ for which bias exceeds the typical
    order of magnitudes of fluctuations, this is relevant to our
    purpose. 
\end{rem}

\subsection{Adaptive Hill estimation}

\label{sec:adapt-hill-estim}
We are now able to characterise  the  performance of the variant of the selection rule defined by \eqref{eq:rule:dk} \citep{MR1632189} with  $r_n = \sqrt{c_3\ln \ln n }$ where~$c_3=2$. 
Let $\ell_n = \lceil c_2 \ln n\rceil$ where $c_2$ is a  constant to be defined below. 

The deterministic sequence of  indices $(k_n(r_n))$  is  defined (for $n$ large enough) by
\begin{equation}\label{def:kn}
 k_n (r_n)= \max \left\{ k  \in \{ \ell_n,\ldots, n\} \colon \sqrt{k}\overline{\eta}(n/k^\delta)
   \leq   \gamma r_n\right\} \, ,
\end{equation}
where $k^\delta= k+\sqrt{2k \ln(1/\delta)}+ 2 \ln(1/\delta).$
The sequence  $(k_n(1))_n$  is defined by  choosing $r_n=1$.
The deterministic sequences $(k_n(1))$  and $(k_n(r_n))$ achieve specific balances between bias and variance. In full generality, because $\overline{\eta}(t)$ is just an upper bound on the conditional bias $b(t)$,  it is difficult to precisely connect  $(k_n(1))$ and 
$(k_n(r_n))$ with the oracle sequence $(k^*_n)$. We call these two sequences the pivotal sequences. 
In the sequel, $k_n$ stands for $k_n(r_n)$. If the context is not clear, we specify $k_n(1)$ or $\ k_n(r_n)$.

Let $1/(2n) < \delta<1/4$.
Recall, from Section \ref{sec:conc-ineq-hill}, that $\xi_n= c_1 \sqrt{\ln \ln (n)}+ c'_1$ and agree on the shorthands 
\begin{eqnarray*}
	z_\delta &= (1+3 r_n/\sqrt{k_n})\left(\xi_n+
  \sqrt{8 \ln  \left(2/\delta \right)} + \frac{\ln  \left(2/\delta \right)}{\sqrt{\ell_n}}\right)			
\end{eqnarray*}
and $\overline{z}_\delta$ which is defined by replacing $k_n$ by $\ell_n$ in the definition of $z_\delta$ ($\overline{z}_\delta$  depends on $n, \delta$ but  not on the sampling distribution). 
In the sequel, $c_2$ is assumed to be chosen so that $r_n+ \overline{z}_\delta \leq 9\sqrt{c_2 \ln(n)}/10$ for $n>1000$  and $ 2/n< \delta <1/4$ ($c_2$ may be chosen not larger than 100).

 The index $\widehat{k}_n$ is selected according to the following rule:
\begin{displaymath}
  \widehat{k}_n =  \max \left\{ k  \in \{ \ell_n,\ldots, n\}  \text{ and } \forall i \in \{\ell_n,\ldots,n\} \, , \left| \widehat{\gamma}(i)-\widehat{\gamma}(k)\right|\leq \frac{\widehat{\gamma}(i) r_n(\delta) }{\sqrt{i}}\right\} 
\end{displaymath}
where $r_n(\delta)= 10(r_n+\overline{z}_\delta)$. The quantity $r_n(\delta)$ scales like $\sqrt{\ln((2/\delta)\ln(n))}$. 
The tail index estimator is $\widehat{\gamma}(\widehat{k}_n)$.

As tail adaptivity has a price (see Theorem ~\ref{thm-kc-lower-bound}), the ratio between the  risk of the data-driven estimator $\widehat{\gamma}(\widehat{k}_n)$ and the risk of the pivotal index $\widehat{\gamma}( k_n (1))$ cannot be upper bounded by a constant factor, let alone by a factor close to $1$. This is why in the next theorem, we compare the risk of the empirically selected index $\widehat{\gamma}(\widehat{k}_n)$ with the risk of the pivotal index $\widehat{\gamma}(k_n)$.

Recall, from Section \ref{sec:conc-ineq-hill}, that 
\begin{displaymath}
\xi_n = c_1 \sqrt{\ln \log_2 n} +c'_1 \quad \text{with }c_1\leq 4 \text{ and } c'_1\leq 34\, . 
\end{displaymath}

\begin{thm}\label{thm:adapt-hill-estim}
Assume the sampling distribution $F \in\mathsf{MDA}(\gamma), \gamma>0$ satisfies the von Mises condition with bounded
von Mises function $\eta$, and $\overline{\eta}(t)= \sup_{s\geq t} |\eta(s)|.$

Let $n > 1000$  is large enough so that $k_n$ (Definition \ref{def:kn}) is well defined.
Then, for $2/n< \delta< 1/4$, with probability larger than $1-3\delta$, 
\begin{displaymath}
\Big| \gamma -   \widehat{\gamma}(\widehat{k}_n) \Big|    \le \left|
  \gamma - \widehat{\gamma}(k_n)  \right|\left(
  1+\tfrac{r_n(\delta)}{\sqrt{k_n}}\right) + \tfrac{r_n(\delta)
}{\sqrt{k_n}}  \gamma \, ,
\end{displaymath}
and, with probability larger than $1-4\delta$, 
\begin{eqnarray}\label{rem:asym:bound:risk}
\Big| \widehat{\gamma}(\widehat{k}_n)-\gamma  \Big| 
&\leq & \tfrac{2 r_n(\delta)}{\sqrt{k_n}} \gamma \left(1 + \alpha(\delta,n) \right) \, ,
\end{eqnarray}
where
\begin{eqnarray*}
	\alpha(\delta,n) & \leq & \frac{r_n}{2\sqrt{\ell_n}} 
	+\frac{\sqrt{ \ln(2/ \delta )}}{ r_n(\delta)}  \left(1 + \frac{3r_n(\delta)}{\sqrt{\ell_n}}\right)^2 \, .
\end{eqnarray*}
\end{thm}

\begin{rem}
For $0<\delta<1/2$, 
\begin{displaymath}
\alpha(\delta,n) = o(1) \quad \text{as $n \to \infty$} \pt
\end{displaymath}
\end{rem}

\begin{rem}
If the bias $b$ is $\rho$-regularly varying  (or equivalently, if the von Mises function $\eta$ or even $\overline{\eta}$ are regularly varying), then, elaborating on Proposition 1 from \citep{MR1632189},  sequences $(k_n^*)$ and  $(k_n(1))$ are connected by
\begin{displaymath}
  \lim_n \frac{k_n(1)}{k_n^*}  = (2|\rho|)^{1/(1+2|\rho|)} 
\end{displaymath}
and their quadratic risk are related by 
\begin{displaymath}
\lim_n \frac{\EXP[(\gamma-\widehat{\gamma}(k_n(1)))^2]}{\EXP[(\gamma-\widehat{\gamma}(k_n^*))^2]}  =\frac{2}{2|\rho|+1}(2|\rho|)^{2|\rho|/(1+2|\rho|)}  \, . 
\end{displaymath}
Moreover, under the second-order assumption, the two pivotal sequences $(k_n(1))$  and $(k_n(r_n))$  are also connected. 

Thus, if the bias is $\rho$-regularly varying, Theorem \ref{thm:adapt-hill-estim} provides us with a connection between the performance of the simple selection rule and the performance of the (asymptotically) optimal choice.

\end{rem}

Recall that one of the main aims of this paper is to derive performance guarantees for the data-driven index selection method $\widehat{k}_n$ without resorting to second-order assumptions that is, without assuming that the von Mises function is regularly varying. The next corollary upper bounds the risk of the preliminary estimator when we just have an upper bound on the bias. 
\begin{cor} \label{cor:adapt-hill-estim-2ndrv}
Assume that, for some $C>0$  and $\rho<0$, for all $t >1$, 
\begin{displaymath}
  \overline{\eta}\left(t\right) \leq C  t^\rho \, .
\end{displaymath}
Then, there exists a constant $\kappa_{C,\delta, \rho}$ depending on $C, \delta$ and $\rho$ such that, with probability larger than $1-4\delta$, 
\begin{eqnarray*}
\Big| \widehat{\gamma}(\widehat{k}_n)-\gamma  \Big| 
&\leq & \kappa_{C,\delta,\rho} \left(\frac{\gamma^2\ln  \left((2/\delta) \ln n\right)}{n}\right)^{|\rho|/(1+2|\rho|)}\left(1+\alpha(\delta,n) \right) 
\end{eqnarray*}
where $\alpha(\delta,n)$  is defined in Theorem \ref{thm:adapt-hill-estim}.
\end{cor}

This meets the information-theoretic lower bound of Theorem \ref{thm-kc-lower-bound}.


\section{Proofs} \label{sec:proofs}
\subsection{Proof of Proposition  \ref{hill:rep}}

This proposition is a straightforward consequence of  R\'enyi's
representation
of order statistics of standard exponential samples. 

As $F$ belongs to $\textsf{MDA}(\gamma)$ and meets the von Mises condition, there exists a function $\eta$ on $(1,\infty)$ with $\lim_{x\to \infty} \eta(x)=0$ such that 
\begin{displaymath}
  U(x) = c x^\gamma \exp \left( \int_1^x \frac{\eta(s)}{s} \mathrm{d}s\right)  \, , 
\end{displaymath}
and
\begin{displaymath}
  U(\mathe^y) =c \exp \left( \int_{0}^y (\gamma +\eta(\mathe^u)) \mathrm{d}u\right) \, . 
\end{displaymath}
Then, 
\begin{eqnarray*}
  \widehat{\gamma}(k) &\dis & \frac{1}{k} \sum_{i=1}^k i\frac{\ln U \left( \mathe^{Y_{(i)}}\right)}{\ln U \left( \mathe^{Y_{(i+1)}}\right)}\\
  &\dis&  \frac{1}{k} \sum_{i=1}^k i\int_{Y_{(i+1)}}^{Y_{(i)}} (\gamma +\eta(\mathe^u)) \mathrm{d}u \\
&\dis & \frac{1}{k} \sum_{i=1}^k i\int_{0}^{E_i/i} (\gamma +\eta(\mathe^{u+Y_{(i+1)}})) \mathrm{d}u \\  
&\dis & \frac{1}{k} \sum_{i=1}^k \int_{0}^{E_i} (\gamma +\eta(\mathe^{\tfrac{u}{i}+Y_{(i+1)}})) \mathrm{d}u  \, . 
\end{eqnarray*}

\subsection{Proof of Proposition \ref{prop:bound-vari-hill}}

Let $Z= k\widehat{\gamma}(k)$.
By the Pythagorean relation, 
\begin{displaymath}
  \var(Z)  =  \EXP \left[ \var\left( Z \mid Y_{(k+1)}\right)\right] + \var\left( \EXP[Z \mid Y_{(k+1)}]\right) \, . 
\end{displaymath}
Representation \eqref{rep:hill:var} asserts that, conditionally on $Y_{(k+1)}$, $Z$ is distributed as a sum of independent, exponentially distributed random variables. Let $E$ be an exponentially distributed random variable.
\begin{eqnarray*}
\lefteqn{  \var\left( Z \mid Y_{(k+1)}= y\right) }\\
&= & k\var\biggl( \gamma E + \int_0^{E} \eta(\mathe^{u +y} )\mathrm{d}u\biggr) \\
& = & k  \gamma^2 + 2 k \gamma\cov\biggl( E, \int_0^{E} \eta(\mathe^{u+y} )\mathrm{d}u\biggr) + \var\biggl(  \int_0^{E}\eta(\mathe^{u+y} )\mathrm{d}u\biggl) \\
& \leq & k  \gamma^2 + 2k\gamma \overline{\eta}(\mathe^y)  + k  \left(\overline{\eta}\left(\mathe^y\right)\right)^2 \, ,
\end{eqnarray*}
where we have used the Cauchy-Schwarz inequality and $ \var\bigl(  \int_0^E \eta(\mathe^{y+u})\mathrm{d}u\bigl) \leq \overline{\eta}(\mathe^{y})^2 .$
Taking expectation with respect to $Y_{(k+1)}$  leads to
\begin{displaymath}
   \EXP \left[ \var\left( Z \mid Y_{(k+1)}\right)\right]  \leq k{\gamma^2}  + 2k\gamma\EXP \left[ \overline{\eta}\left(\mathe^{Y_{(k+1)}} \right)\right] 
+ k{\EXP\left[ \overline{\eta}\left(\mathe^{Y_{(k+1)}} \right)^2\right]} \, .
\end{displaymath}
The last term in the Pythagorean decomposition is also handled using elementary arguments. 
\begin{eqnarray*}
\EXP[Z \mid Y_{(k+1)}] =k \gamma + 
k\int_0^\infty\mathe^{-u}  \eta\left(\mathe^{u +Y_{(k+1)}}\right) \mathrm{d}u  \, .
\end{eqnarray*}
As $Y_{(k+1)}$  is a function of independent exponential random
variables ($Y_{(k+1)}= \sum_{i=k+1}^n  E_i/i$), the variance of $\EXP[Z \mid Y_{(k+1)}]$ may  
be upper bounded using Poincar\'e inequality (Proposition \ref{poincare:expo})
\begin{displaymath}
  \var\left( \EXP[Z \mid Y_{(k+1)}]\right)   
 \leq  4k \EXP \left[ \overline{\eta}\left(\mathe^{Y_{(k+1)}} \right)^2\right]\, . 
\end{displaymath}

In order to derive the lower bound, we first observe that
\begin{displaymath}
   \var(Z)  \geq  \EXP \left[ \var\left( Z \mid
       Y_{(k+1)}\right)\right]  \, .
\end{displaymath}
Now,  using Cauchy-Schwarz inequality again,
\begin{eqnarray*}
 \var\left( Z \mid Y_{(k+1)}= y\right) 
& \geq  & k  \gamma^2 - 2 k \gamma \bigg \lvert \cov\bigl( E, \int_0^E
  \eta(\mathe^{u+y} )\mathrm{d}u\bigr) \bigg \rvert \\
& \geq & k  \gamma^2 -  2 k \gamma  \bigg(\var \bigg(\int_0^E
  \eta(\mathe^{u+y} )\mathrm{d}u\bigg) \bigg)^{1/2} \\
& \geq & k  \gamma^2 - 2k \gamma \overline{\eta}(\mathe^y)  \, . 
\end{eqnarray*}

\subsection{Proof of Theorem  \ref{prp:hill:concentration}}

In the proof of Theorem \ref{prp:hill:concentration}, we will use the next maximal inequality \citep[see][Corollary 2.6]{BoLuMa13}. Recall the definition of $\Gamma_+(v,c)$ (Definition \ref{dfn:sub-gamma}).
\begin{prp}
\label{zlapcor}
Let $Z_1,\ldots,Z_N$ be real-valued random variables belonging
to $\Gamma_+(v,c)$.  Then
\[
\EXP \left[\max_{i=1,\ldots,N} Z_i \right] \leq
\sqrt{2v\ln N}+ c\ln N~.
\]
\end{prp}
Proofs follow a common pattern. In order to check that some random variable is sub-gamma, we rely on its representation as a function of independent
exponential variables and compute partial derivatives, derive convenient upper bounds on the squared Euclidean norm and the supremum norm of the gradient and then invoke Theorem \ref{bernstein:expo}. 

At some point, we will use the next corollary of Theorem \ref{bernstein:expo}. 
\begin{cor}\label{sec:proof-prop-refprp:technique} 
  If $g$ is an almost everywhere differentiable function on $\mathbb{R}$ with uniformly bounded derivative $g'$, then $g(Y_{(k+1)})$
is sub-gamma with variance factor $4\| g' \|^2_\infty/k$ and scale factor $\| g' \|_\infty/k.$
\end{cor}

\begin{proof}[Proof of Theorem  \ref{prp:hill:concentration}]

We start from the exponential representation of Hill estimators (Proposition \ref{hill:rep}) and represent all $\widehat{\gamma}(i)$ as  functions of independent random variables 
$E_1,\ldots,E_{{k}},\ldots, E_J,Y_{({J+1})}$ where the $E_j, 1\le j \le J$, are standard exponentially distributed and $Y_{({J+1})}$ 
is distributed like the $(J+1)$th largest order statistic of an $n$-sample of the standard exponential distribution. We consistently use the notation $Y_{(k)} =  \sum_{j=k}^J \frac{E_j}{j}  +Y_{(J+1)}$, for $1\leq k \leq J$.

\begin{eqnarray*}
i\widehat{\gamma}(i) & = &  \sum_{j=1}^{i} \int_{0}^{E_j} \left(\gamma +\eta(\mathe^{\frac{u}{j}+Y_{({j+1})}})\right) \mathrm{d}u\\
&= &  \sum_{j=1}^{i} \left( \gamma E_j + j \int_{Y_{(j+1)}}^{Y_{(j)}} \eta(\mathe^{v}) \, \mathrm{d}v \right)\, . 
\end{eqnarray*}
Let $i'$ be such that  $0\leq i'< i$,  let us agree on
$\widehat{\gamma}(0)=0.$ Let
\begin{eqnarray*}
 g(E_{i'+1}, \dots, E_{J}) &= & i \widehat{\gamma}(i) - i' \widehat{\gamma}(i') \\
 & = & \sum_{j=i'+1}^i \gamma E_j + \sum_{j=i'+1}^i j 
\int_{Y_{(j+1)}}^{Y_{(j)}} \eta(\mathe^{v})\,  \mathrm{d}v 
  \, .
\end{eqnarray*}
For $i' <p \leq i$, as $\frac{\partial Y_{(j)}}{\partial x_p} = \frac{1}{p}$ for $j\leq p$  and $0$ otherwise, 
\begin{eqnarray*}
 \frac{\partial g}{\partial x_p} 
 & = &  \gamma + \sum_{j=i'+1}^p j \frac{\partial\int_{Y_{(j+1)}}^{Y_{(j)}} \eta(\mathe^{v})\mathd v}{\partial x_p}\\
& = & \gamma + \eta(\mathe^{Y_{(p)}})+ \sum_{j=i'+1}^{p-1} \frac{j}{p} \left( \eta(\mathe^{Y_{(j)}}) - \eta(\mathe^{Y_{(j+1)}}) \right) \\
& = & \gamma + \sum_{j=i'+2}^p \frac{\eta(\mathe^{Y_{(j)}})}{p} + \frac{(i'+1)\eta(\mathe^{Y_{(i'+1)}})}{p} \, . 
\end{eqnarray*}
This entails that, for $i'<p\leq i$, 
\begin{equation}\label{eq:part:deriv:1}
  \left|  {\frac{\partial g}{\partial x_p}} \right| \leq \gamma + 
  \sum_{j = 1}^p \frac{\overline{\eta}(\mathe^{Y_{(p\vee i'+1)}})}{p} 
  \leq  \gamma + \overline{\eta}(\mathe^{Y_{(p)}}) \, . 
\end{equation}
For $i <p \leq  J$, 
\begin{eqnarray*}
  \frac{\partial g}{\partial x_p} 
& = & \sum_{j=i'+1}^{i}  j \frac{\partial \int_{Y_{(j+1)}}^{Y_{(j)} } \eta(\mathe^{v})\, \mathrm{d}v }{\partial x_p}  \\
& = & \sum_{j=i'+1}^{i} \frac{j}{p} \left( \eta(\mathe^{Y_{(j)}}) - \eta(\mathe^{Y_{(j+1)}}) \right) \\
& = & \frac{1}{p} \left(\sum_{j=i'+2}^i \eta(\mathe^{Y_{(j)}}) + (i'+1)\eta(\mathe^{Y_{(i'+1)}}) - i\eta(\mathe^{Y_{(i+1)}})\right)\\
& = & \frac{1}{p}  \left( \sum_{j=i'+1}^i \left( \eta(\mathe^{Y_{(j)}})- \eta(\mathe^{Y_{(i+1)}})\right)  
+ i'\left( \eta(\mathe^{Y_{(i'+1)}})- \eta(\mathe^{Y_{(i+1)}})\right)   \right) \, .
\end{eqnarray*}
This is enough to entail that, for $i<p\leq k$,
\begin{equation}\label{eq:part:deriv:2}
  \left|  {\frac{\partial g}{\partial x_p}} \right|  \leq 
     \frac{2i}{p} \overline{\eta}(\mathe^{Y_{(p)}}) \, . 
  \end{equation}
All in all, for $1 \leq p \leq k$,
\begin{displaymath}
   \left|  {\frac{\partial g}{\partial x_p}} \right| \leq  (\gamma + \overline{\eta}(T)) \vee 2 \overline{\eta}(T) \leq \gamma + 2\overline{\eta}(T)\, .  
\end{displaymath}

\paragraph*{Proof of i)}
An upper bound on the variance factor for $i\widehat{\gamma}(i)$, conditionally on $T$,   is obtained by specialising to the case $i'=0$ and using \eqref{eq:part:deriv:1} and \eqref{eq:part:deriv:2} as well as the monotonicity of $\overline{\eta}$, 
\begin{eqnarray*}
   \sum_{p=1}^{J}  \left|  \frac{\partial g}{\partial x_p} \right|^2
  & \leq  &  \sum_{p=1}^i \left(\gamma + \overline{\eta}(\mathe^{Y_{(p)}}) \right)^2 + \sum_{p=i+1}^J \frac{4i^2}{p^2} 
\overline{\eta}(\mathe^{Y_{(p)}})^2 \\
& \leq & 
  {i} \left( \left( \gamma + \overline{\eta}(T)\right)^2 + 4  
  \left(\overline{\eta}(T)\right)^2 \right) \leq i \left( \gamma + 3\overline{\eta}(T)\right)^2\, . 
\end{eqnarray*}

  Using Theorem \ref{bernstein:expo} conditionally on $T= \exp(Y_{(J+1)})$, we realise  that $\sqrt{i}(\widehat{\gamma}(i)-\EXP\left[ \widehat{\gamma}(i)\mid T \right])$
  is sub-gamma on both sides with variance factor not larger than $4 \left( \gamma + 3\overline{\eta}(T)\right)^2$ and scale factor not larger than 
  $\gamma + 2\overline{\eta}(T)$. This yields
\begin{flalign*}
\mathbb{P}&\left\{ \lvert \widehat{\gamma}(i)-\EXP\left[ \widehat{\gamma}(i)\mid T \right] \rvert \geq  \frac{\gamma+3\overline{\eta}\left(T\right)}{\sqrt{i}} \left( \sqrt{8s} +\frac{s}{\sqrt{i}}\right) \mid T\right\} \leq 2 \mathe^{-s}  \, . 
\intertext{Taking expectation on both sides, this implies that }
\mathbb{P}&\left\{ \lvert \widehat{\gamma}(i)-\EXP\left[ \widehat{\gamma}(i)\mid T\right] \rvert \geq  \frac{\gamma+3\overline{\eta}\left(T\right)}{\sqrt{i}} \left( \sqrt{8s} +\frac{s}{\sqrt{i}}\right) \right\} \leq 2 \mathe^{-s}  \, . 
\end{flalign*}

\paragraph*{Proof of ii)} 
\label{par:proof_of_ii_}
The proof of the upper bound on $\mathbb{E}[Z^a \mid T]$ in  Statement
ii) from Theorem \ref{prp:hill:concentration} relies on standard
chaining techniques from the theory of empirical processes and uses repeatedly the concentration Theorem \ref{bernstein:expo} for smooth functions of independent exponential random variables and the maximal inequality for sub-gamma random variables (Proposition \ref{zlapcor}).

For general $i'$, the variance factor for $i\widehat{\gamma}(i)- i'\widehat{\gamma}(i')$ is upper bounded by 
\begin{displaymath}
   (i-i')  \left(\gamma +  \overline{\eta}(T) \right)^2 + 
    \sum_{p=i+1}^J \frac{i^2}{p^2}
   \left(2 \overline{\eta}(T) \right)^2 \leq 
  (i-i') \left(\gamma +  \overline{\eta}(T) \right)^2 + i (2\overline{\eta}(T))^2 \, . 
\end{displaymath} 
Let $u$ be such that $J\overline{\eta}(u)^2 \leq \gamma^2 r_n^2$ where $r_n=\sqrt{c_3\ln \ln n}$ with $c_3=2$. Now, as we assume, in the sequel, that $T>u$, we may use the next upper bound for the variance 
factor of $i \widehat{\gamma}(i)-i' \widehat{\gamma}(i')$ (conditionally on $Y_{(J+1)}$),\
\begin{equation*}
4 \gamma^2 \left( (i-i') \Big(1+ \frac{ r_n}{\sqrt{J}} \Big)^2 + 4 r_n^2 \right) \, .    
\end{equation*} 
Recall that 
\begin{displaymath}
  Z^a=  \max_{  \ell \leq i\leq k} \sqrt{i} \left| \widehat{\gamma}(i) -\EXP \left[\widehat{\gamma}(i) \mid Y_{(k+1)}\right] \right|\, .
\end{displaymath}
As it is commonplace in the analysis of normalised empirical processes
\citep[see ][and references therein]{vandegeer:2000,gine:koltchinskii:2006,massart:2003}, we peel the index set over 
which the maximum is computed. 

Let $\mathcal{L}_n = \{ \lfloor \log_2 (\ell) \rfloor, \ldots, 
\lfloor \log_2({k})\rfloor  \}$ and, for all $j \in \mathcal{L}_n$, $\mathcal{S}_j = \{ \ell \vee 2^j , \ldots ,  k\wedge 2^{j+1} -1\} $. Define $Z^a_j$ as 
\begin{displaymath}
  Z^a_j  = \max_{i \in \mathcal{S}_j} \sqrt{i} \left|\widehat{\gamma}(i) -\EXP \left[\widehat{\gamma}(i)\mid Y_{(k+1)}\right] \right|\, .
\end{displaymath}
Then,  
\begin{eqnarray*}
\EXP [Z^a\mid Y_{(k+1)}] & =  & \EXP [\max_{j\in \mathcal{L}_n} Z^a_j \mid Y_{(k+1)}] \\
& \leq & \EXP [\max_{j\in \mathcal{L}_n} (Z^a_j-\EXP [Z^a_j\mid  Y_{(k+1)}] )\mid Y_{(k+1)}] ] + \max_{j\in \mathcal{L}_n} \EXP[ Z^a_j\mid  Y_{(k+1)}] ] \, . 
\end{eqnarray*}
We now derive  upper bounds  on both summands by resorting to the maximum inequality for sub-gamma random variables (Proposition \ref{zlapcor}). 
We first bound $\EXP[ Z^a_j\mid  Y_{(k+1)}]$, for $j\in \mathcal{L}_n$.  

Note that direct invocation of Lemma \ref{zlapcor} and Statement i) shows that 
\begin{equation}\label{eq:poor:bound}
	\mathbb{E} [Z^a_j \mid Y_{(k+1)}] \leq 2 \gamma(1+3r/\sqrt{J})\left( \sqrt{8  j \ln(2) } + j \ln(2)\right) \, . 
\end{equation}
This bound will be useful for handling small values of $j$. For $j\leq 11$,  $\sqrt{8  j \ln(2) } + j \ln(2)\leq 16$.

We now handle generic $j$ using chaining.  Fix $j\in \mathcal{L}_n$,
  \begin{displaymath}
  \max_{i\in \mathcal{S}_j} \sqrt{i} \left| \widehat{\gamma}(i) - \EXP [\widehat{\gamma}(i) \mid Y_{(k+1)}] \right| 
  \leq \frac{1}{2^{j/2}} \max_{i \in \mathcal{S}_j} \,  i\left| \widehat{\gamma}(i) - \EXP [\widehat{\gamma}(i) \mid Y_{(k+1)}] \right|  \, . 
\end{displaymath}
In order to alleviate notation, let $W(i) = i \left( \widehat{\gamma}({i}) -\EXP [ \widehat{\gamma}({i})  \mid Y_{(k+1)}] \right)$,  for $i\in \mathcal{S}_j$. For $i\in \mathcal{S}_j$, let
\begin{displaymath}
i = 2^{j} +\sum_{m=1}^j b_m 2^{j-m} \ \text{ where $b_m \in \{0,1\}$}
\end{displaymath} 
be the binary expansion of $i$. Then, for $h\in \{0, \ldots, j\}$, let $\pi_h(i)$ be defined by 
\begin{displaymath}
\pi_{h}(i)=2^j + \sum_{m=1}^h b_m 2^{j-m} \,
\end{displaymath}
so that $\pi_j(i)=i$, $\pi_0(i)=2^j$ and $0\leq \pi_{h+1}(i)-\pi_h(i)\le 2^{j-h-1}$. 

Using the fact that $W(\pi_0(i))$ does not depend on $i$ and that 
$$
\esp \left[W(\pi_0(i)) \mid Y_{(k+1)} \right]=0\, ,
$$ 
we obtain
\begin{eqnarray*}
  \lefteqn{\EXP \left[ \max_{i \in \mathcal{S}_j}  i\left(\widehat{\gamma}(i) -\EXP \left[\widehat{\gamma}(i)\mid Y_{(k+1)}\right] \right)\mid Y_{(k+1)} \right]}\\
& = & \EXP  \left[ \max_{i \in \mathcal{S}_j} W(i) \mid Y_{(k+1)} \right] \\
& = & \EXP  \left[ \max_{i \in \mathcal{S}_j} W(\pi_j(i)) -W(\pi_0(i))\mid Y_{(k+1)} \right]  \\
& = & \EXP  \left[ \max_{i \in \mathcal{S}_j} \sum_{h=0}^{j-1} (W(\pi_{h+1}(i)) -W(\pi_h(i)))\mid Y_{(k+1)}  \right] \\
& \leq & \sum_{h=0}^{j-1} \EXP \left[ \max_{i \in \mathcal{S}_j}  (W(\pi_{h+1}(i)) -W(\pi_h(i))) \mid Y_{(k+1)} \right] \, .
\end{eqnarray*}
Now, for each $h\in \{0,\ldots,j-1\}$, the maximum  is taken over $2^h$ random variables which are sub-gamma with variance factor 
\[
4 \gamma^2 \times \left(  2^{j-h-1}\Big(1+  \frac{r_n}{\sqrt{J}}\Big)^2 +4 r_n^2\right)
\] 
and scale factor $(\gamma+ 2\overline{\eta}(T))\leq \gamma (1+ 2r_n/\sqrt{J})$. By Proposition~\ref{zlapcor}, since $i \in \mathcal S_j$, 
\begin{eqnarray*}
  \lefteqn{\EXP \left[ \max_{i \in \mathcal{S}_j}  i\left(\widehat{\gamma}(i) -\EXP \left[\widehat{\gamma}(i)\mid Y_{(k+1)}\right] \right)\mid Y_{(k+1)} \right]}\\
&\leq &   \gamma \sum_{h=0}^{j-1} \left(\left(1+ \frac{r_n}{\sqrt{J}}\right) \sqrt{8 h  2^{(j-h-1)}\ln 2 } + \sqrt{32 h \ln 2} r_n+ \left(1+ \frac{2r_n}{\sqrt{J}}\right) h \ln 2\right) \\
& \leq &  \gamma \left(1+ \frac{2r_n}{\sqrt{J}}\right)  \left( 2^{(j-1)/2}  4.15 \sqrt{8 \ln 2 }+ \frac{2}{3 } \sqrt{32 c_2 } \ln(2) j^{2}+  \frac{j(j-1)}{2} \ln 2\right) 
\end{eqnarray*}
where we have used 
\begin{eqnarray*}
 4.15 	& \geq & \sum_{h=0}^\infty \sqrt{h 2^{-h}} \\
r_n &\leq & \sqrt{c_2 \ln(2)}j^{1/2} \text{ as } r_n=\sqrt{c_3\ln\ln(n)}, j+1\geq \log_2(c_2 \ln (n)) \, .
\end{eqnarray*}
For $j\geq 12$, 
\begin{eqnarray*}
{\EXP \left[ \max_{i \in \mathcal{S}_j}  i\left(\widehat{\gamma}(i) -\EXP \left[\widehat{\gamma}(i)\mid Y_{(k+1)}\right] \right)\mid Y_{(k+1)} \right]}
& \leq &   17 \gamma \, 2^{j/2}  \left(1+ \frac{2r_n}{\sqrt{J}}\right)\, . 
\end{eqnarray*}


Finally, for all $j \in \mathcal{L}_n$, 
\begin{displaymath}
  \EXP [Z^a_j\mid Y_{(k+1)}] \leq 34 \,  \gamma \,  \left(1+ \frac{3r_n}{\sqrt{J}}\right)   \,. 
\end{displaymath}

In order to prove Statement ii), we check that, for each $j \in \mathcal{L}_n$,  $Z^a_j$ is sub-gamma on the right-tail  with variance factor at most $4\left(\gamma+ 3 \overline{\eta}(T) \right)^2$ and scale factor not larger than $\left(\gamma+ 3 \overline{\eta}(T)\right)/\sqrt{\ell}$. Under the von Mises condition (Definition \ref{dfn:vmises:cond}), the sampling
distribution is absolutely continuous with respect to Lebesgue measure. For almost every sample, the maximum defining $Z^a_j$ is attained at a single index $i\in \mathcal{S}_j$. Starting again from the exponential representation
and repeating the computation of partial derivatives, we obtain the desired bounds. 

By Proposition \ref{zlapcor}, 
\begin{eqnarray*}
\lefteqn{\EXP \left[ \max_{j \in \mathcal{L}_n} (Z^a_j -\EXP [Z^a_j \mid Y_{(k+1)}]) \mid Y_{(k+1)}\right]} \\
& \leq &\left(   \sqrt{8\ln |\mathcal{L}_n|}  + \frac{ \ln |\mathcal{L}_n|}{\sqrt{\ell}}\right)\left(\gamma + 3 \overline{\eta}(T)   \right)  \\
 & \leq& 4 \sqrt{\ln |\mathcal{L}_n|} \left(\gamma + 3
   \overline{\eta}(T)   \right)  \\
  & \leq & 4 \, \gamma\, \sqrt{\ln |\mathcal{L}_n|}  \left(1 + \frac{3r_n}{\sqrt{J}}   \right)  \, 
\end{eqnarray*}
where we have used $\ln |\mathcal{L}_n| \leq {\ln(\log_2(n))}\leq {\ln(n)}\leq \ell  $, for $n \geq 2$.  
Combining the different bounds leads to the upper bound on $\mathbb{E}[Z^a \mid T]$.
\end{proof}

\subsection{Proof of Theorem \ref{thm:adapt-hill-estim}}

Throughout this proof, let
\begin{align*}
  T_n &=\exp \left(Y_{(k_n+1)}\right)\\
  \xi_n & = c_1 \sqrt{\ln \log_2 n} +c'_1 \quad \text{where }c_1, c'_1 \text{ are defined in Section \ref{sec:conc-ineq-hill}},  \\
z_\delta &= (1+3 r_n/\sqrt{k_n})\left(\xi_n+
  \sqrt{8 \ln  \left(2/\delta \right)} + \frac{\ln  \left(2/\delta \right)}{\sqrt{\ell_n}}\right) \\
\overline{z}_\delta  &= (1+3 r_n/\sqrt{\ell_n})\left(\xi_n+
  \sqrt{8 \ln  \left(2/\delta \right)} + \frac{\ln  \left(2/\delta \right)}{\sqrt{\ell_n}}\right) 
\, . 
\end{align*}
Let us define the events $E1$ 
and $E_2$  as
\begin{eqnarray*}
 E_1&=& \Big\{ c_2 \ln n \leq i \leq k_n, \, \sqrt{i} \left| \widehat{\gamma}(i) -\EXP [\widehat{\gamma}(i)\mid T_n]\right|\leq  \gamma z_\delta \Big\}\, , \\
  E_2 &=& \Big\{ T_n  \geq \frac{n}{k_n^\delta} \Big\} \text{ with } k_n^\delta= k_n+2{\ln \left(1/\delta \right)}+\sqrt{2k_n \ln (1/\delta)}\pt
\end{eqnarray*}

The fact that $\mathbb{P}(E_2) \geq 1-\delta$ follows from the following reformulation of 
 Proposition 4.3 from \citep{boucheronthomas2012} (a proof is given in Appendix \ref{proof:prop:right:tail:order:stat}).
 \begin{prp}\label{prop:right:tail:order:stat}
  For $\delta\in (0,1)$, with probability larger that $1-\delta$,
 \[
  \exp(Y_{(k+1)}) \geq  \frac{n}{k^\delta}  \qquad \text{ with } k^\delta=  k + 2\ln(1/\delta) +\sqrt{2 k \ln(1/\delta)}.
  \]
  where $Y_{(k+1)}$ is the $(k+1)$th largest order statistic of an exponential sample of size $n$.
 \end{prp}

By Theorem \ref{prp:hill:concentration}, $\mathbb{P}(E_1 \mid E_2) \geq 1-\delta$. Hence, 
 the event $E_1 \cap E_2$   has probability at least $(1- \delta)^2 \geq 1-2\delta$.

Under $ E_2$, \begin{enumerate}[i)]
\item $\overline{\eta}(T_n)\leq \gamma r_n/\sqrt{k_n}$.  
\item for all $\ell_n\leq i \leq {k_n}$, $|\gamma- \mathbb{E}[\widehat{\gamma}(i)\mid T_n]|\leq \overline{\eta}(T_n).$
\end{enumerate}

The first step of the proof consists in checking that under $E_1 \cap E_2$, the selected index is not smaller than $k_n$. It suffices to check that for all $i,k \text{ such that }  \ell_n\leq i < k < k_n$, 
\begin{eqnarray*}
\sqrt{i}\left| \widehat{\gamma}(i) - \widehat{\gamma}(k)\right| \le r_n(\delta) \widehat{\gamma}(i) \, .
\end{eqnarray*}
For all $i \in \{ \ell_n , \ldots, k_n\}$ , 
\begin{eqnarray*}
  \gamma -\widehat{\gamma}(i) &\leq & \left| \gamma - \EXP[\widehat{\gamma}(i)\mid T_n]\right| + \left|\widehat{\gamma}(i)  -\EXP[\widehat{\gamma}(i)\mid T_n]\right|\\
 &\leq & \overline{\eta}(T_n)  + \frac{\gamma z_\delta}{\sqrt{i}}  \\
&\leq & \frac{\gamma r_n}{\sqrt{k_n}} + \frac{\gamma z_\delta}{\sqrt{i}} 
\end{eqnarray*}
so that
\begin{displaymath}
  \frac{\widehat{\gamma}(i)}{\gamma} \geq  1  -\frac{r_n+z_\delta}{\sqrt{i}} \, . 
\end{displaymath}
Meanwhile, for all $i,k$, 
\begin{eqnarray*}
\lefteqn{\left| \widehat{\gamma}(i)-\widehat{\gamma}(k)\right| }\\
&\leq &\underbrace{\Big| \widehat{\gamma}(i)-\EXP[\widehat{\gamma}(i)\mid T_n]\Big|}_{\textsc{(i)}}+ \underbrace{\left|  \EXP[\widehat{\gamma}(i)-\widehat{\gamma}(k)\mid T_n]\right|}_{\textsc{(ii)}} + \underbrace{\left| \widehat{\gamma}(k)-\EXP[\widehat{\gamma}(k)\mid T_n]\right|}_{\textsc{(iii)}} \, . 
\end{eqnarray*}
Under $E_1 \cap E_2$, for $\ell_n \leq i < k \leq k_n$, 
\begin{displaymath}
\textsc{(i)}+\textsc{(iii)} \le \gamma z_\delta \left( \frac{1}{\sqrt{i}} + \frac{1}{\sqrt{k}}\right) \le \frac{2\gamma}{\sqrt{i}} z_\delta \pt
\end{displaymath} 
Under $E_2$,
\begin{eqnarray*}
 \textsc{(ii)}
& \leq &  \left| \EXP[\widehat{\gamma}(i)-{\gamma}\mid T_n]\right| + \left| \EXP[ \gamma-\widehat{\gamma}(k)\mid T_n]\right| \\
& \leq &  2 \overline{\eta}(T_n) \\
& \leq &  2 \gamma r_n /\sqrt{k_n} \, . 
\end{eqnarray*}

Plugging upper bounds on \textsc{(i)}, \textsc{(ii)} and \textsc{(iii)}, it comes that, under $E_1 \cap E_2$, for all $k \le k_n-1$ and for all $i \in \{ \ell_n, \ldots,k\} $,
\begin{eqnarray*}
\sqrt{i} \frac{\left| \widehat{\gamma}(i)-\widehat{\gamma}(k)\right|}{\gamma} &\leq &  
2 z_\delta + 2 r_n  \, . 
\end{eqnarray*}

In order to warrant that, under 
$E_1 \cap E_2$, %
for all $k < k_n$ and for all $i$ such that $c_2 \ln n \le i \le k$, $\sqrt{i} \left| \widehat{\gamma}(i)-\widehat{\gamma}(k)\right| \le r_n(\delta) \widehat{\gamma}({i})$, it is enough to have
\begin{displaymath}
  2 (z_\delta +  r_n)  \leq r_n(\delta) \left(1 -\frac{r_n+z_\delta}{\sqrt{i}}\right)  \, . 
\end{displaymath}
The last inequality holds because 
\begin{displaymath}
	2 (z_\delta +  r_n)  \leq r_n(\delta) \left(1 -\frac{r_n+z_\delta}{\sqrt{\ell_n}}\right)
\end{displaymath}
by definition of $r_n(\delta)$. 

Hence, with probability larger than $(1- \delta)^2$, $E_1 \cap E_2$ is realised, and under $E_1 \cap E_2$,
$\widehat{k}_n \geq k_n$.

We now check that if $\widehat{k}_n\geq k_n$, the risk of $\widehat{\gamma}(\widehat{k}_n)$ is not much larger than the risk of $\widehat{\gamma}(k_n)$. 

\begin{eqnarray*}
\left| \gamma - \widehat{\gamma}(\widehat{k}_n) \right| &\le& \left|
  \gamma - \widehat{\gamma}(k_n)  \right| + \left|
  \widehat{\gamma}(k_n)   -
  \widehat{\gamma}(\widehat{k}_n) \right| \\
& \leq & \left|
  \gamma - \widehat{\gamma}(k_n)  \right| +
\tfrac{r_n(\delta)
  \widehat{\gamma}(k_n)}{\sqrt{k_n}} 
 \pt
\end{eqnarray*}
Therefore, under $E_1\cap E_2$, 
\begin{equation}\label{ineq:oracle}
\left| \gamma - \widehat{\gamma}(\widehat{k}_n) \right|  \le  \left|
  \gamma - \widehat{\gamma}(k_n)  \right|\left(
  1+\tfrac{r_n(\delta)}{\sqrt{k_n}}\right) + \tfrac{r_n(\delta)
 \gamma}{\sqrt{k_n}} \pt
\end{equation}

Now, consider the event 
$E_1 \cap E_2 \cap E_3$
with 
\begin{displaymath}
E_3=\Bigg\{ \sqrt{k_n}\left| \widehat{\gamma}(k_n) -\EXP [\widehat{\gamma}(k_n)\mid T_n]\right| \leq \left(\gamma+3 \overline{\eta}(T_n) \right) \Big( \sqrt{8 \ln  \left(2/\delta \right)} + \frac{\ln  \left(2/\delta \right)}{\sqrt{k_n}}\Big)\Bigg\} \pt
\end{displaymath} 
Since, $\mathbb{P} (E_3 \mid E_2 ) \ge 1-\delta$, thanks to Statement i) from Theorem \ref{prp:hill:concentration}, the event 
$E_1 \cap E_2 \cap E_3$ has probability at least $(1- \delta)(1-2 \delta)\geq 1-3 \delta$. 

Then, by definition of $k_n$,  under $E_2$, 
\[
\left|
  \gamma -\mathbb{E}[\widehat{\gamma}({k}_n)\mid T_n]\right|\leq \overline{\eta}(T_n) \leq \gamma r_n/\sqrt{k_n}  \, .
  \]
  Hence, under   $E_2 \cap E_3$, 
\begin{eqnarray*}
  \Big|\widehat{\gamma}({k}_n)-\gamma\Big| &\leq & |\gamma -\mathbb{E}[\widehat{\gamma}({k}_n)\mid T_n]|+ |
  \widehat{\gamma}({k}_n) -\mathbb{E}[\widehat{\gamma}({k}_n)\mid T_n]| \\
  & \leq & \frac{\gamma}{\sqrt{k_n}}\left(r_n +\left(1+ \frac{3r_n}{\sqrt{k_n}}\right)\Big( \sqrt{8 \ln  \left(2/\delta \right)} + \frac{\ln  \left(2/\delta \right)}{\sqrt{k_n}}\Big) \right)
\, . 
\end{eqnarray*}
Therefore, plugging this bound into \eqref{ineq:oracle}, with probability larger than $1-3\delta$, 
\begin{eqnarray*}
\lefteqn{\Big| \widehat{\gamma}(\widehat{k}_n)-\gamma  \Big|  }\\
&\leq & \frac{\gamma}{\sqrt{k_n}}  \left(r_n(\delta) + \left( r_n + \left(1+ \frac{3r_n}{\sqrt{k_n}}\right)\left(\frac{\sqrt{8 \ln \left(2/\delta \right)}}{2}
 + \frac{\ln \left(2/\delta \right)}{2\sqrt{k_n}}\right)\right)\left(
  1+\frac{r_n(\delta)}{\sqrt{k_n}}\right) \right) \\
  & \leq & \frac{2 \gamma r_n(\delta)}{\sqrt{k_n}} (1+ \alpha(\delta,n))\, ,
\end{eqnarray*}
where 
\begin{eqnarray*}
	\alpha(\delta,n) & = &  \frac{r_n}{2\sqrt{\ell_n}} 
	+\frac{\sqrt{ \ln(2/ \delta )}}{ r_n(\delta)}  \left(1 + \frac{3r_n(\delta)}{\sqrt{\ell_n}}\right)^2  \, . 
\end{eqnarray*}
\subsection{Proof of Corollary \ref{cor:adapt-hill-estim-2ndrv}} 
If, for some $C>0$  and $\rho<0$,
\begin{displaymath}
\overline{\eta}(t)
  \leq C t^\rho \, , 
\end{displaymath}
then, by the definition of $k_n$,
\begin{displaymath}
  \frac{\gamma r_n}{\sqrt{k_n+1}} \leq C \left( \frac{n}{(k_n+1)^\delta} \right)^\rho \, ,
\end{displaymath}
which entails that
\begin{displaymath}
  \frac{\gamma r_n}{\sqrt{k_n+1}} \leq C \left( \frac{n}{\left(\sqrt{k_n+1}+\sqrt{2 \ln(1/\delta)}\right)^2} \right)^\rho \, .
\end{displaymath}
Solving this inequality leads to
\begin{displaymath}
	\sqrt{k_n+1}  \geq \sqrt{\left(\frac{ \gamma r_n}{C}\right)^{2/(1+2|\rho|)}n^{2|\rho|/(1+2|\rho|)}} - \frac{2|\rho|\sqrt{2 \ln(1/\delta)}}{1+2|\rho|} 
\end{displaymath}
and finally to
\begin{displaymath}
	k_n \geq \frac{1}{2} \left( \frac{\gamma r_n}{C}\right)^{1/(1+2|\rho|)} n^{|\rho|/(1+2|\rho|)}-
	 2\left(\frac{2|\rho|\sqrt{2 \ln(1/\delta)}}{1+2|\rho|}\right)^2
	-1 \, . 
\end{displaymath}
Thus, for sufficiently large $n$,  there exists a constant $c$ depending on $\rho,\delta$ such that 
\begin{displaymath}
  \sqrt{k_n}  \geq \left( \frac{\gamma r_n}{c}\right)^{1/(1+2|\rho|)} n^{|\rho|/(1+2|\rho|)} \, .
\end{displaymath}
Starting from Equation \eqref{rem:asym:bound:risk} of Theorem \ref{thm:adapt-hill-estim}, with probability $1-3\delta$, 
\begin{eqnarray*}
\Big| \widehat{\gamma}(\widehat{k}_n)-\gamma  \Big| 
&\leq & 16\gamma\sqrt{\frac{2\ln( (2/\delta) \log_2 n)}{k_n}}  \left(1 + \alpha(\delta,n)\right) \, ,
\end{eqnarray*}
and, there exists a constant $\kappa_{C, \delta,\rho}$, depending on $C, \delta$ and $\rho$, such that 
\begin{eqnarray*}
 \sqrt{\frac{\ln( (2/\delta) \log_2 n)}{k_n}}  \le \kappa_{C,\delta,\rho} \gamma^{-1/(1+2|\rho|)} \left(\frac{\ln((2/\delta) \log_2 n)}{n}\right)^{|\rho|/(1+2|\rho|)} \pt
\end{eqnarray*}
Hence, with probability larger than $1-4\delta$, 
\begin{eqnarray*}
\Big| \widehat{\gamma}(\widehat{k}_n)-\gamma  \Big| 
&\leq &\kappa_{C, \delta,\rho} \left(\frac{\gamma^2\ln((2/\delta) \log_2 n)}{n}\right)^{|\rho|/(1+2|\rho|)} \left(1+\alpha(\delta,n) \right)\pt
\end{eqnarray*}

\section{Simulations}
\label{sec:simulations}
Risk bounds like Theorem  \ref{thm:adapt-hill-estim} and Corollary
\ref{cor:adapt-hill-estim-2ndrv} are conservative. For all practical
purposes, they are just meant
to be reassuring guidelines. In this numerical section, we intend to
shed some light on the following issues:
\begin{enumerate}
\item Is there a reasonable way to calibrate the threshold
  $r_n(\delta)$  used in the definition of $\widehat{k}_n$? How does
  the method perform if we choose $r_n(\delta)$  close to $\sqrt{2 \ln
    \ln (n)}$? 
\item  How large is  the ratio between the risk of
  $\widehat{\gamma}(\widehat{k}_n)$  and the  risk of $\widehat{\gamma}(k_n^*)$
  for moderate sample sizes?    
\end{enumerate}

The finite-sample performance of the data-driven index selection
method described and analysed in Section \ref{sec:adapt-hill-estim}
has been assessed by Monte-Carlo simulations. Computations have been
carried out in \texttt{R} using packages \texttt{ggplot2} \citep{wickham2009ggplot2},
\texttt{knitr}, \texttt{foreach}, \texttt{iterators}, \texttt{xtable}
and  \texttt{dplyr} \citep[see][for a modern account of the R
environment]{wickham2014advanced}.  To get into the details, we investigated the performance of index selection methods on samples of sizes $1 000, 2 000$
and $10 000$ from the collection of distributions listed in Table \ref{tab-risk-profiles}. The list comprises the following distributions
\begin{enumerate}[i)]
\item Fr\'echet distributions $F_\gamma(x)= \exp(x^{-1/\gamma})$ for
  $x>0$ and $\gamma \in \{ 0.2, 0.5,1\}$.
\item Student distributions  $t_\nu$ with $\nu \in \{ 1,2,4,10\}$  degrees of freedom.
\item The log-gamma distribution with density proportional to $(\ln (x))^{2-1}x^{-3-1}$, which means $\gamma=1/3$ and $\rho=0$.
\item The L\'evy distribution  with density $\sqrt{{1}/{(2\pi)}} \exp({ -\tfrac{1}{2 x}})/x^{3/2}$, $\gamma = 2$ and $\rho =-1 $ (this is the distribution of 
$1/X^2$ when $X \sim \mathcal{N}(0,1)$).
\item The $H$ distribution is defined by $\gamma=1/2$  and von Mises function equal to $\eta(s)=(2/s) \ln 1/s$. This 
distribution satisfies the second-order regular variation condition with  $\rho= -1$ but does not satisfy Condition \eqref{hall:condition}.
\item Two Pareto change point distributions with distribution functions   $$ \overline{F}(x) =  x^{-1/\gamma'} \mathbb{1}_{\{1\leq x \leq \tau\}} +
      \tau^{-1/\gamma'} (x/\tau)^{-1/\gamma} \mathbb{1}_{\{x \geq \tau\}}  \,
      $$ and  $\gamma\in \{ 1.5, 1.25\}$,  $\gamma' = 1$, and thresholds $\tau$ adjusted in such a way that they correspond to quantiles of order $1-1/15$ and $1-1/25$, respectively.
\end{enumerate}
Fr\'echet, Student, log-gamma distributions were used as benchmarks by \citep{MR1632189}, \citep{MR1821820} and \citep{carpentierkim2014}.  

Table~\ref{tab-risk-profiles}, which is complemented by
Figure~\ref{fig:risk-comp}, describes the difficulty of tail index
estimation from samples of the different distributions. 
Monte-Carlo estimates of the standardised root mean square error (\textsc{rmse}) of 
Hill estimators  
$$\mathbb{E}\left[\left(\widehat{\gamma}(k)/\gamma-1\right)^2\right]^{1/2}$$
are represented as functions of the number of order statistics $k$ for samples of  size $10 000$ from the sampling distributions. 
All curves exhibit a common pattern:  for small values of $k$, the \textsc{rmse} is dominated by the variance
term and scales like $1/\sqrt{k}$. Above a threshold that depends on the
sampling distribution but that is not completely characterised by the
second-order regular variation index, the \textsc{rmse} grows at a rate that may
 reflect the second-order regular variation property (if any) of the distribution.
Not too surprisingly, the three Fr\'echet distributions exhibit the same
risk profile. The three curves are almost undistinguishable.  The Student distributions
illustrate the impact of the second-order parameter on the difficulty of the index selection problem.
For sample size $n=10 000$, the optimal index for $t_{10}$  is smaller than $30$, it is smaller than the usual recommendations. 
For such moderate sample sizes, distribution $t_{10}$ seems as hard to handle as the $\log$-gamma distribution which usually 
fits in the Horror Hill Plot gallery. The $1/2$-stable  L\'evy
distribution and the $H$-distribution behave very differently. Even though they both have second-order parameter $\rho$
equal to $-1$, the $H$ distribution seems almost as challenging as the $t_4$  distribution while the L\'evy distribution looks much easier than
the Fr\'echet distributions. The Pareto change point distributions exhibit an abrupt transition.  

\begin{table}[h]
\centering
\caption{Estimated oracle index $k_n^*$  and standardised
  \textsc{rmse}
  $\mathbb{E}[(\gamma-\widehat{\gamma}(k_n^*))^2]^{1/2}/\gamma$ for
  benchmark distributions. Estimates were computed from $5000$
  replicated experiments on samples of size $10 000.$}
\label{tab-risk-profiles}
\begin{tabular}{lcccc}
  \toprule
d.f. & $\gamma$ & $\rho$ & $k_n^*$& RMSE \\ 
  \midrule
$F_{0.2}$& 0.2 & 1.0 & 1132 & 3.7e-02 \\ 
 $ F_{0.5}$  & 0.5 & 1.0 & 1145 & 3.6e-02 \\ 
 $ F_1$ & 1.0 & 1.0 & 1155 & 3.6e-02 \\ 
 $t_1$ & 1.0 & 2.0 & 1161 & 3.3e-02 \\ 
 $t_2$ & 0.5 & 1.0 & 341 & 6.5e-02 \\ 
  $t_ 4$ & 0.2 & 0.5 & 77 & 1.6e-01 \\ 
  $t_{10}$& 0.1 & 0.2 & 15 & 5.3e-01 \\ 
  H & 0.5 & 1.0 & 130 & 1.1e-01 \\ 
  log-gamma & 0.3 & 0.0 & 213 & 1.6e-01 \\ 
   Stable & 2.0 & 1.0 & 3172 & 2.0e-02 \\ 
 Pcp & 1.5 & 0.3 & 943 & 3.3e-02 \\ 
  Pcp (bis) & 1.2 & 0.2 & 593 & 4.2e-02 \\ 
   \bottomrule
\end{tabular}

\end{table}

\begin{figure}[h]\centering

  \includegraphics[width=.8\textwidth]{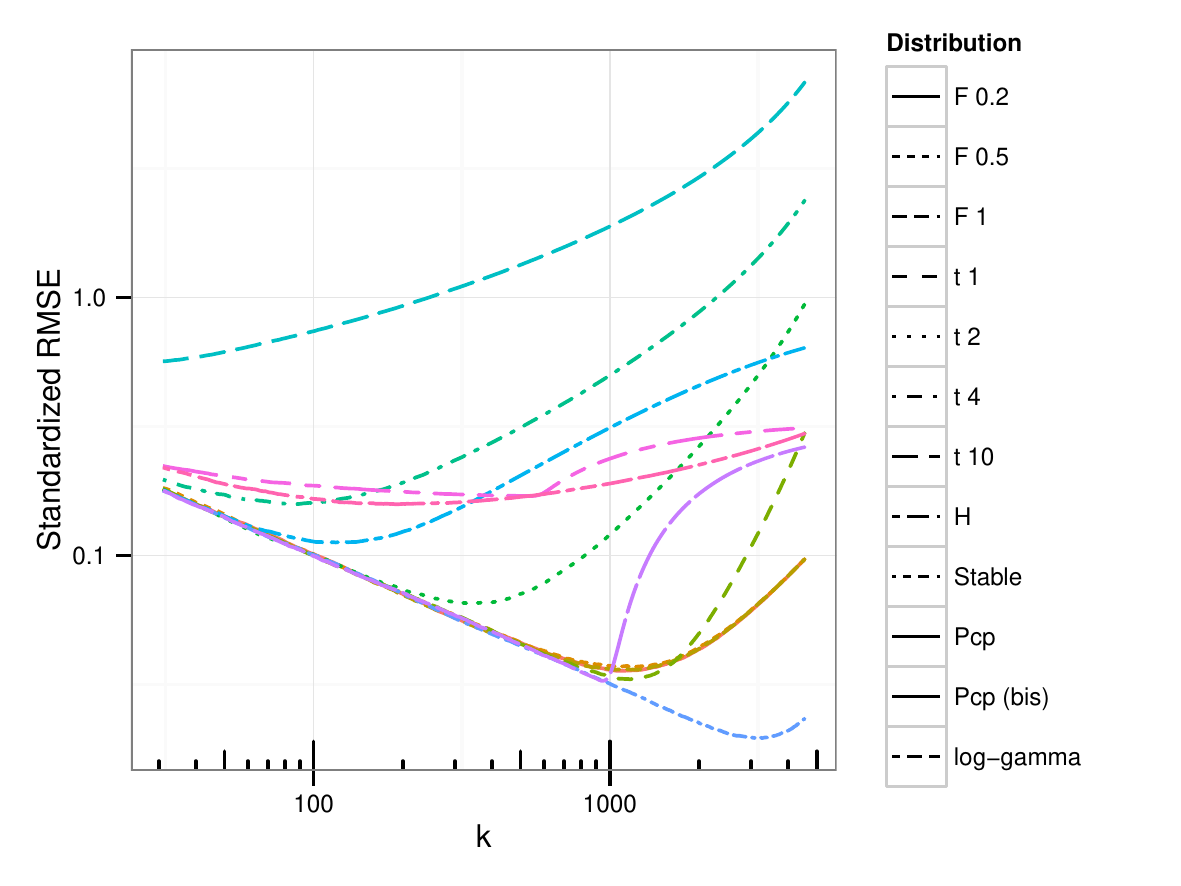}
 \caption{
 Monte-Carlo estimates of the standardised root mean square error (\textsc{rmse}) of 
Hill estimators 
as a  function of the number of order statistics $k$ for samples of  size $10 000$ from the sampling distributions. 
}

  \label{fig:risk-comp}
\end{figure}

Index $\widehat{k}_n(r_n)$ was computed according to the following rule
{\footnotesize
\begin{equation}\label{eq:1}
   \widehat{k}_n(r_n) =  \min \left\{ k  \colon 30  \leq k \leq n \text{ and } \exists i \in \{30,\ldots,k\} \, , \left| \widehat{\gamma}(i)-\widehat{\gamma}(k)\right|>\frac{r_n \widehat{\gamma}(i)}{\sqrt{i}}\right\} -1 
\end{equation}}
with $r_n= \sqrt{c \ln \ln n}$  where $c=2.1$ unless otherwise specified. 

The Fr\'echet, Student, $H$ and stable distributions all fit into the
framework considered by \citep{MR1632189}. They provide a favorable
ground for comparing the performance of the optimal index selection
method described by \citet{MR1632189} which attempts to take advantage
of the second-order regular variation property and the performance of
the simple selection rule described in this paper. 

Index $\widehat{\gamma}(\widehat{k}^{\textsc{dk}}_n)$ was computed following the recommandations from Theorem 1 and discussion in \citep{MR1632189}
\begin{equation}\label{eq:2}
\widehat{k}^{\textsc{dk}}_n = (2 |\widehat{\rho}|+1)^{-1/|\widehat{\rho}|}   \left( 2 \widehat{\rho} \widehat{\gamma}\right)^{1/(1+2|\hat{\rho}|)} \left(\frac{ \widehat{k}_n(r_n^\zeta)}{(\widehat{k}_n(r_n))^\zeta}  \right)^{1/(1-\zeta)}
\end{equation}
where $\widehat{\rho}$ should  belong to a consistent family of estimators of $\rho$ (under a second-order regular variation assumption),  $\widehat{\gamma}$
should be a preliminary estimator of $\gamma$ such as
$\widehat{\gamma}(\sqrt{n})$,  $\zeta=.7$, and $r_n= 2n
^{1/4}$. Following the advice from \citep{MR1632189}, we replaced $
|\widehat{\rho}|$ by $1$.  Note that the method for computing
$\widehat{k}^{\textsc{dk}}_n$ depends on a variety of tunable
parameters.

 Comparison between performances of
  $\widehat{\gamma}(\widehat{k}_n(r_n))$  and
  $\widehat{\gamma}(\widehat{k}_n^{\textsc{dk}})$ are reported in
  Tables \ref{tab:ratios:1} and \ref{tab:ratios:2}. For each
  distribution from Table \ref{tab-risk-profiles},  for sample sizes
  $n=1000, 2000, \text{ and } 10 000$, $5000$ experiments were replicated.
 As pointed out in \citep{MR1632189}, on the sampling distributions
 that satisfy a second-order regular variation property, carefully
 tuned $\widehat{k}_n^{\textsc{dk}}$ is able to take advantage of
 it. Despite its computational and conceptual simplicity and the
 fact that it is almost parameter free,  the  estimator
 $\widehat{\gamma}(\widehat{k}_n(r_n))$ only suffers a moderate loss
 with respect to the oracle. When $|\rho|=1$, the observed ratios are
 of the same order as $(2\ln \ln n)^{1/3}\approx 1.65$. Moreover, 
whereas $\widehat{\gamma}(\widehat{k}_n^{\textsc{dk}})$ behaves
erratically when facing Pareto change point distributions,
$\widehat{\gamma}(\widehat{k}_n(r_n))$
behaves consistently. 

\begin{table}[ht]
\centering
\caption{Ratios between median selected indices $\widehat{k}_n(r_n)$ (Lepski),
  $\widehat{k}^{\textsc{dk}}_n$ (Drees-Kaufmann) and estimated oracle
  index $k_n^*$.} 
\label{tab:ratios:1}
\begin{tabular}{lccccccc}
  \toprule
  \multirow{2}{*}{d.f.} & \multirow{2}{*}{$\gamma$} & \multicolumn{3}{c}{$\widehat{k}^{\textsc{dk}}_n/k_n^*$} &
                                                         \multicolumn{3}{c}{$\widehat{k}_n(r_n)/k_n^* $}\\
\cmidrule(lr){3-5} \cmidrule(l){6-8}
 & & $n=1000$ & 2000 & 10000 & 1000 & 2000 & 10000 \\ 
  \midrule
$F_{0.2}$ & 0.2 & 0.61 & 0.67 & 0.94 & 2.94 & 2.97 & 3.47 \\ 
$  F_{0.5}$ & 0.5 & 1.12 & 1.18 & 1.45 & 2.90 & 2.87 & 2.91 \\ 
  $F_ 1$ & 1 & 1.76 & 2.05 & 2.32 & 2.90 & 3.10 & 2.93 \\ 
  $t_1$ & 1 & 1.33 & 1.55 & 1.98 & 2.03 & 2.16 & 2.16 \\ 
  $t _2$ & 0.5 & 1.00 & 0.99 & 0.91 & 3.05 & 3.06 & 2.96 \\ 
  $t_4$ & 0.25 & 1.27 & 1.28 & 1.18 & 5.62 & 5.50 & 5.30 \\ 
  $t_{10}$ & 0.1 & 2.00 & 1.54 & 2.28 & 13.87 & 10.92 & 14.12 \\ 
  H & 0.5 & 0.41 & 0.35 & 0.30 & 5.14 & 4.97 & 4.96 \\ 
  Stable & 2 & 0.97 & 0.95 & 1.04 & 1.43 & 1.41 & 1.55 \\ 
  Pcp & 1.5 & 1.85 & 0.45 & 0.15 & 1.32 & 1.21 & 1.10 \\ 
  Pcp (bis) & 1.25 & 3.29 & 3.03 & 2.45 & 1.83 & 1.50 & 1.22 \\ 
  log-gamma & 0.33 & 5.13 & 7.71 & 12.41 & 10.50 & 12.99 & 12.40 \\ 
   \bottomrule
\end{tabular}
\end{table}

\begin{table}[ht]

\centering
\caption{Ratios between median \textsc{rmse}   and median
  optimal \textsc{rmse}.} 
\label{tab:ratios:2}
\begin{tabular}{lccccccc}
  \toprule
 \multirow{2}{*}{d.f.} & \multirow{2}{*}{$\gamma$} & \multicolumn{3}{c}{$\textsc{rmse}(\widehat{\gamma}(\widehat{k}^{\textsc{dk}}_n))/\textsc{rmse}(\widehat{\gamma}(k_n^*))$} &
                                                         \multicolumn{3}{c}{$\textsc{rmse}(\widehat{\gamma}(\widehat{k}_n(r_n)))/\textsc{rmse}(\widehat{\gamma}(k_n^*))$}\\
\cmidrule(lr){3-5} \cmidrule(l){6-8}
 & & $n=1000$ & 2000 & 10000 & 1000 & 2000 & 10000 \\ 
  \midrule
$F_{0.2}$ & 0.2 & 1.12 & 1.12 & 1.02 & 2.06 & 2.26 & 2.69 \\ 
 $ F_{0.5}$ & 0.5 & 1.03 & 1.03 & 1.14 & 2.12 & 2.23 & 2.70 \\ 
  $F_1$ & 1 & 1.22 & 1.31 & 1.59 & 2.07 & 2.23 & 2.64 \\ 
  $t_1$ & 1 & 1.26 & 1.34 & 1.74 & 2.31 & 2.39 & 3.11 \\ 
  $t_2$ & 0.5 & 1.11 & 1.08 & 1.05 & 2.06 & 2.09 & 2.20 \\ 
  $t_4$ & 0.25 & 1.10 & 1.07 & 1.04 & 1.85 & 1.81 & 1.84 \\ 
  $t _{10}$ & 0.1 & 1.10 & 1.09 & 1.08 & 1.76 & 1.72 & 1.64 \\ 
  H & 0.5 & 1.28 & 1.37 & 1.48 & 2.15 & 2.18 & 2.12 \\ 
  Stable & 2 & 1.01 & 0.99 & 0.98 & 1.99 & 2.52 & 3.60 \\ 
  Pcp & 1.5 & 4.25 & 1.66 & 2.52 & 2.50 & 2.68 & 3.63 \\ 
  Pcp (bis) & 1.25 & 3.38 & 4.47 & 7.45 & 2.43 & 2.56 & 3.10 \\ 
  log-gamma & 0.33 & 1.23 & 1.28 & 1.39 & 1.45 & 1.43 & 1.37 \\ 
   \bottomrule
\end{tabular}

\end{table}

Figure \ref{fig:pcp-risk-plot} concisely describes the behaviour of the
two index selection methods on samples from the Pareto change point
distribution with parameters $\gamma=1.5, \gamma'=1$ and threshold
$\tau$ corresponding to the $1-1/15$ quantile. The plain line
represents the standardised \textsc{rmse} of Hill estimators as a
function of selected index. This figure contains the superposition of
two density plots corresponding to $\widehat{k}^{\textsc{dk}}_n$ and
$\widehat{k}(r_n)$. The density plots were generated from $5000$
points with coordinates $(\widehat{k}(r_n),
|\widehat{\gamma}(\widehat{k}(r_n))/\gamma-1|)$ and $5000$ points with
coordinates  $(\widehat{k}^{\textsc{dk}}_n,
|\widehat{\gamma}(\widehat{k}^{\textsc{dk}}_n)/\gamma-1|)$. The
contoured and well-concentrated density plot corresponds to the
performance of $\widehat{\gamma}(\widehat{k}_n)$. The diffuse tiled
density plot corresponds to the performance of
$\widehat{k}^{\textsc{dk}}_n$. Facing Pareto change point samples, the
two selection methods behave differently. Lepski's rule detects 
correctly an abrupt change at some point and selects an index slightly
above that point. As the conditional bias varies sharply around the
change point, this slight over estimation of the correct index still
results in a significant loss as far as \textsc{rmse} is concerned. The
Drees-Kaufmann rule, fed with an a priori estimate of the second-order
parameter, picks out a much smaller index, and suffers a larger excess risk.

\begin{figure}[h]\centering
\includegraphics[width=.8\textwidth]{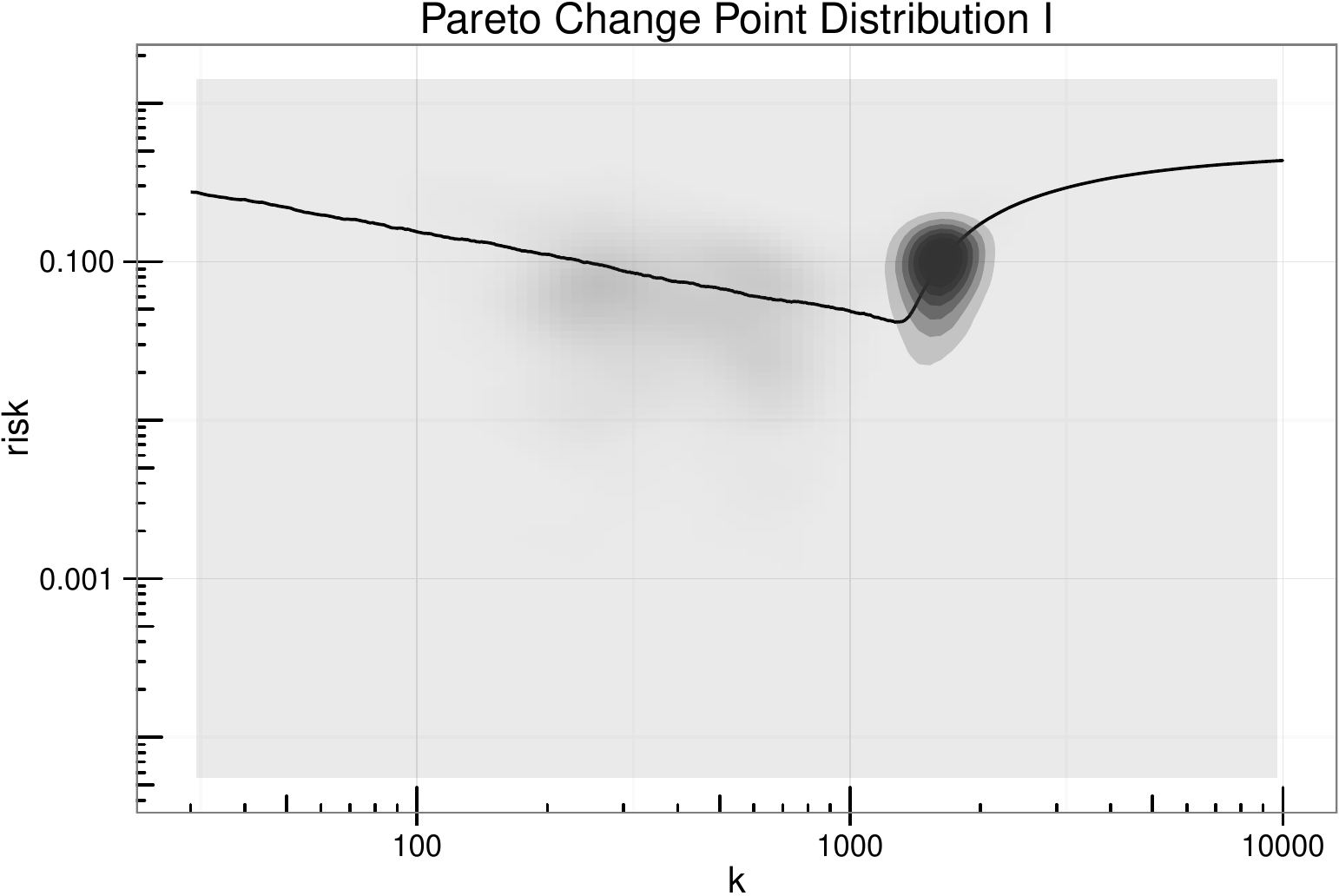}
  \caption{Risk plot for samples of size $10 000$ from  the Pareto
    change point distribution with parameters $\gamma=1.5, \gamma'=1$
    and threshold $\tau$ corresponding to the $1-1/15$ quantile. The
    concentrated density plot corresponds to points
    $(\widehat{k}(r_n),|\widehat{\gamma}(\widehat{k}(r_n))/\gamma-1| )$.}
  \label{fig:pcp-risk-plot}
\end{figure}

\flushleft{\textbf{Acknowledgement}} The authors are thankful to the editor and the referees for their careful reading and valuable suggestions, which led to detect an error and to a improved version of the paper. 
\bibliographystyle{abbrvnat}


\appendix

\section{Calibration of the preliminary selection rule}
\label{sec:calibr-prel-select}
\sloppy{
\citet{DarErd56} establish  (among
other things) that letting $Z_n$ denote $\sup_{k\leq n} \sum_{i=1}^k (E_i-1)/\sqrt{k}$, where $E_i$, $1\leq i \leq k$, are independent exponentially distributed random variables, the sequence $\sqrt{2\ln\ln n} \left(Z_n-\sqrt{2\ln\ln n} -   \ln\ln\ln (n)/(2\sqrt{2\ln\ln n}) \right)$
 converges in distribution towards a translated Gumbel distribution.  In other words, asymptotically,  $Z_n$ behaves almost
like the maximum of $\ln n $ independent standard Gaussian random variables. }

\section{Proof of Corollary \ref{prp:hazard:conc:ineg}}
Let $Z=g(E_1,\ldots,E_n)=\left(U \circ \exp\right)\left(\sum_{i=1}^k E_i/i\right)$. Then, 

\begin{displaymath}
  \vert \partial_i g \vert \leq \frac{1}{i} \sup_{x} \frac{1}{h(x)}, \, \text{ for } i\geq k,  
\end{displaymath}
and 
\begin{displaymath}
  \| \nabla g\|^2 =  \sum_{i=k}^n \frac{1}{i^2}
  \frac{1}{(h \circ\ g)^2} \pt
\end{displaymath}
Let $c<1$, then for all $\lambda, 0 \leq \lambda \leq  c\,( k \inf_x h(x)) $, 
\begin{displaymath}
  \ln \EXP \mathe^{\lambda (Z-\EXP Z)} \leq \frac{4/k\left(1+1/k\right) \EXP [1/h(Z)^2]\lambda^2 }{2(1-c) } \, . 
\end{displaymath}
Now, start from the first statement in Theorem \ref{bernstein:expo},
\begin{eqnarray*}
    \operatorname{Ent} \left[ \mathe^{\lambda (Z-\EXP Z)}\right]  &\leq &
        \frac{2\lambda^2}{1-c} \EXP \left[ \mathe^{\lambda (Z-\EXP Z)} \| \nabla
          f\|^2\right] \\
 &=&   \frac{4\lambda^2}{2(1-c) } \frac{1}{k} \left( 1+ \frac{1}{k}\right) \EXP \left[ \frac{\mathe^{\lambda (Z-\EXP Z)}}{h(Z)^2}\right] \\
& \leq & \frac{4\lambda^2}{2(1-c) } \frac{1}{k} \left( 1+ \frac{1}{k}\right) \EXP \left[ {\mathe^{\lambda (Z-\EXP Z)}}\right] \EXP \left[ \frac{1}{h(Z)^2}\right]
\end{eqnarray*}
where the last inequality follows from Chebychev negative association inequality. Hence,  
  \begin{displaymath}
  \frac{\mathrm{d} }{\mathrm{d}\lambda} \left[ \frac 1 \lambda \ln \EXP \mathe^{\lambda (Z-\EXP Z)} \right] = \frac{\operatorname{Ent} \left[ \mathe^{\lambda (Z-\EXP Z)}\right]}{\lambda^2
      \EXP \left[ \mathe^{\lambda (Z-\EXP Z)}\right]}  \leq
    \frac{1}{2(1-c)}  \frac{4}{k} \left( 1+\frac{1}{k}\right)  \EXP \left[ \frac{1}{h(Z)^2}\right] \, .
  \end{displaymath}
This differential inequality is readily solved and leads to the corollary.

\section{Proof of Abelian Proposition \ref{sec:abel-vari-hill}}

The proof proceeds by classical arguments. In the sequel, we use the almost sure representation argument. Without loss
of generality, we assume that all the random variables live on the same probability space, and that, for any intermediate sequence $(k_n)$, $\sqrt{k_n} (Y_{(k_n+1)}-\ln (n/k_n))$ converges almost surely towards a standard Gaussian random variable. Complemented with dominated convergence arguments, the next lemma will be the key element of the proof.

\begin{lem}
\label{lem:rvplus}
Let $\eta \in \emph{\textsf{RV}}_\rho, \rho \leq 0$ and $Y_{(k_n+1)}$ be the $(k_n+1)$th largest order statistic of a standard exponential sample, then, for any intermediate sequence $(k_n)$ and $u>0$,
\begin{displaymath}
\lim_{n \to \infty} \frac{ \eta(\mathe^{u+Y_{(k_n+1)}})}{\eta (n/k_n)} = \mathe ^{\rho u} \ \text{p.s} \pt
\end{displaymath}
\end{lem}

\begin{proof}
Note that
\begin{eqnarray*}
\frac{ \eta  (\mathe^{u+Y_{(k_n+1)}})}{\eta (n/k_n)} &= &\frac{ \eta\Big((n/k_n)\mathe^{u+Y_{(k_n+1)}-\log(n/k_n)}\Big)}{\eta(n/k_n)}  \, . 
\end{eqnarray*}
Then, the result follows since $Y_{(k_n+1)} - \log(n/k_n)\overset{\text{p.s}} {\longrightarrow} 0$ and the convergence $\eta(tx)/ \eta (t) \rightarrow x^\rho$ is locally uniform on $(0,\infty)$. 
\end{proof}

In order to secure dominated convergence arguments, we will use Drees's improvement of Potter's inequality \citep[see][page 369]{dehaanferreira2006}. For every $\epsilon,\delta>0$, there exists $t_0=t_0(\epsilon,\delta)$ such that, for $t,tx\geq t_0$,
  \begin{equation}
  \label{potter:drees}
\left\vert\eta(tx)/\eta(t)-x^\rho\right\vert \leq x^\rho \epsilon\max(x^\delta,x^{-\delta})\, .
\end{equation}

To prove Proposition \ref{sec:abel-vari-hill}, we start from Representation \eqref{rep:hill:var}: 
\begin{displaymath}
\widehat{\gamma}(k_n) = \frac{1}{k_n} \sum_{i=1}^{k_n} \int_0^{E_i} \left( \gamma + \eta \left( \mathe^{u+Y_{(k_n+1)} }\right)\right) \mathrm{d}u \pt
\end{displaymath} 
By the Pythagorean relation, 
\begin{displaymath}
  \var(\widehat{\gamma}(k_n))  =  \var\left( \EXP[\widehat{\gamma}(k_n) \mid Y_{(k_n+1)}]\right)+\EXP \left[ \var\left( \widehat{\gamma}(k_n) \mid Y_{(k_n+1)}\right)\right] \, , 
\end{displaymath}
so that 
\begin{eqnarray*}
\lefteqn{  \frac{k_n\var(\widehat{\gamma}(k_n)) -\gamma^2}{\eta \left(
      n/k_n\right)}}
\\   &= &\frac{k_n \var\left(
      \EXP[\widehat{\gamma}(k_n) \mid Y_{(k_n+1)}]\right)}{\eta(n/k_n) 
  }+  k_n \EXP \left[ \frac{\var\left( \widehat{\gamma}(k_n) \mid Y_{(k_n+1)}\right)-\frac{\gamma^2}{k_n}}{\eta(n/k_n) } \right] \, . 
\end{eqnarray*}
The second summand can be further decomposed using \eqref{rep:hill:var}.
\begin{eqnarray*}
  \lefteqn{  \frac{k_n\var(\widehat{\gamma}(k_n)) -\gamma^2}{ \eta \left(
      n/k_n\right) }}\\
&= & \underbrace{\frac{k_n\var\left( \EXP[\widehat{\gamma}(k_n) \mid
    Y_{(k_n+1)}]\right)}{\eta(n/k_n)}}_{\textsc{(i)}} \\
&& +\underbrace{\eta 
  (\tfrac{n}{k_n})\EXP \left[\var \left[ \int_0^E
     \frac{\eta(\mathe^{u+Y_{(k_n+1)}})}{ \eta (n/k_n) }\mathd u
\mid Y_{(k_n+1)}   \right]\right]}_{(\textsc{ii})}\\
&&  + \underbrace{2\gamma \EXP \left[\cov \left[
     E,\int_0^E \frac{\eta(\mathe^{u+Y_{(k_n+1)}})}{\eta (n/k_n)}\mathd u
   \mid Y_{(k_n+1)}\right]\right]}_{(\textsc{iii})}  . 
\end{eqnarray*}
We check that \textsc{(i)} and \textsc{(ii)} tend to $0$ and then that
\textsc{(iii)} converges towards a finite limit. 

Fix $\epsilon,\delta>0$ and define $M=\sup\{ \eta(t), t \leq t_0 \}$.\\
Let $A_{n}$ denote the event $\{Y_{(k_n+1)}>\ln
t_0(\epsilon,\delta)\}$. For $n$ such that $ \ln (n/k_n) \leq 2 \ln t_0$, as
$Y_{(k_n+1)}$ is sub-gamma with variance factor $1/k_n$,
\begin{displaymath}
\mathbb{P}\{ A^c_n\} \leq \exp\left(-
  k_n(\ln (n/k_n))^2/8\right) \, .
\end{displaymath}

We first check that \textsc{(ii)} tends to
$0$. Let $n$ be such that
  $n/k_n\geq t_0$ and $W_n$ denote the random variable $ Y_{(k_n+1)} -\ln
  \left(n/k_{n}\right)$. Note that, for $0\leq \lambda\leq k_n/2$,
  \begin{displaymath}
    \EXP \mathe^{\lambda |W_n|} \leq 2 \mathe^{\frac{\lambda^2}{k_n}} \, . 
  \end{displaymath}

Using Jensen's inequality and Fubini's Theorem,
\begin{eqnarray*}
\lefteqn{\EXP \left[ \var \left[ \int_0^E
     \frac{\eta(\mathe^{u+Y_{(k_n+1)}})}{\eta (n/k_n)}\mathd u  \mid Y_{(k_n+1)}
   \right]\right]}\\ 
& \leq & \EXP \left[\EXP \left[ E \int_0^E \left(
     \frac{\eta(\mathe^{u+Y_{(k_n+1)}})}{\eta (n/k_n)}\right)^2\mathd u \mid Y_{(k_n+1)}
   \right]\right] \\
& = & \int_0^\infty \mathe^{-v}  v \int_0^v  \EXP \left[\left(
     \frac{\eta(\mathe^{u+Y_{(k_n+1)}})}{\eta (n/k_n)}\right)^2
 \right] \mathd
   u \mathrm{d}v \\
& = & \int_0^\infty \mathe^{-v}  v \int_0^v  \EXP \left[\left(
     \frac{\eta(\mathe^{u+W_n}n/k_n)}{\eta (n/k_n)}\right)^2
 \right] \mathd
   u \mathrm{d}v 
\end{eqnarray*}
We now apply Potter's inequality (\ref{potter:drees}) on the event $A_n$ with $t=n/k_n>t_0$ and $tx=\mathe^{u+Y_{(k_n+1)}}>t_0, u>0$ : 
\begin{eqnarray*}
  \lefteqn{\EXP \left[ \var \left[ \int_0^E
     \frac{\eta(\mathe^{u+Y_{(k_n+1)}})}{\eta (n/k_n)}\mathd u  \mid Y_{(k_n+1)}
   \right]\right]} \\
& \leq  & \int_0^\infty \mathe^{-v}  v \int_0^v  \EXP \left[  \mathbb{1}_{A_{n}}
     \mathe^{2\rho (u+W_n)}
  \left(  1+\epsilon  \mathe^{\delta (u+|W_n|)} \right)^2 + \mathbb{1}_{A_n^c} \frac{M^2}{\eta(n/k_n)^2}
 \right] \mathd
   u \mathrm{d}v \\
& \leq & \int_0^\infty \mathe^{-v}  v \int_0^v  \EXP \left[  
     \mathe^{2\rho W_n}
  2 \left( 1+\epsilon^2  \mathe^{2\delta (u+|W_n|)} \right)
 \right] \mathd
   u \mathrm{d}v+ \frac{2M^2}{\eta(n/k_n)^2}\EXP \mathbb{1}_{A_n^c} \,
   . 
\end{eqnarray*}
The first summand has a finite limit thanks to Lemma \ref{lem:rvplus}. The second summand converges to
$0$ as $\EXP \mathbb{1}_{A_n^c}$ tends to $0$  exponentially fast
while $1/\eta(n/k_n)^2$ tends to infinity algebraically fast.

Bounds on \textsc{(i)} are easily obtained, using Jensen's Inequality and Poincar\'e Inequality. 
\begin{eqnarray*}
\frac{k_n \var\left( \EXP[\widehat{\gamma}(k_n) \mid
    Y_{(k_n+1)}]\right)}{\eta(n/k_n) } 
&= &  \frac{k_n \var\left( \int_0^\infty {\eta \left( \mathe^{u+Y_{(k_n+1)}} \right)} \mathe^{-u} \mathd u\right)}{\eta(n/k_n) } 
\\
&\le& 4 \eta(n/k_n) \esp\left[ \left( \int_0^\infty \frac{\eta \left(
        \mathe^{u+Y_{(k_n+1)}} \right)}{\eta(n/k_n)} \mathe^{-u}
    \mathd u \right)^2\right] \\
&\leq & 
4 \eta(n/k_n) \esp\left[ \int_0^\infty \left( \frac{\eta \left(
        \mathe^{u+Y_{(k_n+1)}} \right)}{\eta(n/k_n)}\right)^2
  \mathe^{-u} \mathd u \right] \, . 
\end{eqnarray*}
Using the line of arguments as for handling the limit of \textsc{(ii)},
we establish that \textsc{(i)} converges to $0$.

We now check that  \textsc{(iii)} converges towards a finite
limit. Note that 
\begin{eqnarray*}
 \lefteqn{\EXP \left[\cov \left[
     E,\int_0^E \frac{\eta(\mathe^{u+Y_{(k_n+1)}})}{\eta (n/k_n)}\mathd u
   \mid Y_{(k_n+1)}\right]\right] }\\
& = &  \EXP  \left[
     (E-1)\int_0^E \frac{\eta(\mathe^{u+Y_{(k_n+1)}})}{\eta (n/k_n)}\mathd u
   \right] \, . 
\end{eqnarray*}
By Lemma \ref{lem:rvplus}, for almost every $u>0$, 
$$(E-1) \frac{\eta(\mathe^{u+W_n} n/k_n)}{\eta (n/k_n)} \underset{n \to \infty}{\longrightarrow} (E-1) \mathe^{\rho u}\, , $$ 
and
\begin{eqnarray*}
 \lefteqn{\lvert E-1 \rvert  \int_0^E \left |\frac{\eta(\mathe^{u+W_n} n/k_n)}{\eta (n/k_n)}\right | \mathd u }\\
 &\le & | E-1|  \int_0^E
   \mathe^{\rho (u+W_n)} \left(1+ \epsilon \mathe^{\delta (u+|W_n|)}\right) \mathd u +\mathbb{1}_{A_n^c} E |E-1| \frac{M}{|\eta(n/k_n)|}\pt 
\end{eqnarray*}
   The first term is finite as the integral of a continuous function on a compact. \\
   Thus, 
$$(E-1) \int_0^E \frac{\eta(\mathe^{u+W_n} n/k_n)}{\eta (n/k_n)}\mathd
u \rightarrow_n (E-1) \int_0^E \mathe^{\rho u} \mathd
u= (E-1)\frac{\mathe^{\rho E} -1}{\rho} \, .$$ 
The expected value of the last random variable is
$1/(1-\rho)^2$. 

 We check that, for sufficiently large $n$, 
\begin{eqnarray*}
\lefteqn{\esp\left[ |E-1| \int_0^E \frac{|\eta(\mathe^{u+W_n} n/k_n)|}{|\eta (n/k_n)|}\mathd u  \right]}\\
 &\leq  & \esp \left[ | E-1|  \int_0^E
   \mathe^{\rho (u+W_n)} \left(1+ \epsilon
     \mathe^{\delta (u+|W_n|)}\right) + \mathbb{1}_{A_n^c} |E-1| \frac{M}{|\eta(n/k_n)|} \mathd u \right] \\
 &\leq & \EXP \left[\mathe^{\rho W_n} \left( 2 + \frac{\epsilon}{\delta(1-\delta)^2} \mathe^{\delta|W_n|} \right)\right]
 +\frac{M}{|\eta(n/k_n)|}  \EXP \mathbb{1}_{A_n^c}  \\
&\leq &  4 \mathe^{\frac{\rho^2}{k_n}} +  \frac{2 \epsilon}{\delta(1-\delta)^2} \mathe^{\frac{(\delta-\rho)^2}{k_n}} + \frac{M}{|\eta(n/k_n)|}  \EXP \mathbb{1}_{A_n^c}  \pt
\end{eqnarray*} 
We now way conclude by dominated convergence that 
$$ \textsc{(iii)} \underset{n \to \infty}{\longrightarrow} \frac{2\gamma}{(1-\rho)^2} \pt $$

\section{Proof of Proposition \ref{prop:right:tail:order:stat}}
\label{proof:prop:right:tail:order:stat}

The proof of Proposition 4.3 from \citep{boucheronthomas2012} 
yields that, with probability larger than $1-\delta$, for $0\le z$, 
\begin{displaymath}
\mathbb{P} \left\{ \exp \left(Y_{(k+1)} \right)\ge  \frac{n}{k}\mathe^{- z} \right \} \ge 1- \exp \left( - 2 k \sinh\left( {z}/{2} \right)^2 \right) \, .
\end{displaymath}
We may choose  $z=2 \arsinh(\sqrt{{\ln\left(1/\delta \right)}/{2k}})$ and notice  that $\arsinh(x)= \ln (x +\sqrt{1+x^2})$. This yields
\begin{displaymath} 
\mathe^{z} \le  1+ 2 \frac{\ln\left( 1/\delta  \right)}{k} + \sqrt{\frac{2 \ln(1/\delta)}{k}}
\quad \text{ and } \quad  \exp \left( - 2 k \sinh\left( {z}/{2} \right)^2 \right)= \delta \, . 
\end{displaymath}



\section{Revisiting the lower bound on adaptive estimation error}\label{proof:lower:bound}

Lower bounds on tail index estimation error \citep{MR1833966,drees1998optimal, Nov14, carpentierkim2014}
are usually constructed by defining sequences of  local models around a pure Pareto distribution with shape parameter $\gamma_0$.
 When deriving lower bounds for the estimation error under constraints like $\overline{\eta}$ is regularly varying, the elements of the local model for sample size $n$  may be defined by
\begin{displaymath}
  U_{n,h}(t) =  t^{\gamma_0+d_nh(0)} \exp \int_1^t d_n\frac{h(c_n/s)-h(0)}{s} \mathrm{d}s  
\end{displaymath}
where $h$ is square integrable over $[0,1]$, $d_n \rightarrow 0$, $nd_n^2/c_n \rightarrow 1$ \citep{MR1833966}. 
The sequences $d_n$ and $c_n$ are chosen in such a way that $d_n \left| h(c_n/s)-h(0)\right|= | \overline{\eta}(s)|$ satisfies the required constraint.
If the local alternatives are Pareto change point distributions as in \citep{Nov14} and \citep{carpentierkim2014},
$h(x)=\mathbb{1}_{\{ x \leq 1\}}$, $c_n= \tau_n^{1/\gamma_0}$.  \citet{MR1833966} explores a richer collection of local alternatives in order to fit into the theory of weak convergence of local experiments. 

In order to explore adaptivity as in \citep{carpentierkim2014}, it is necessary to handle simultaneously a collection of sequences $(d_n,c_n)_n$ corresponding to different rates of decay of the von Mises function. The difficulty of estimation is connected with the difficulty of distinguishing  alternatives with different tail indices that is, with the hardness of a multiple hypotheses testing problem. In order to lower bound the 
testing error, \citeauthor{carpentierkim2014} chose to use Fano's Lemma \citep[see]{CoTh91}. Using Fano's Lemma requires bounding 
the Kullback-Leibler divergence between the different local alternatives which is not as easy as bounding the divergence between a Pareto change point distribution and a pure Pareto distribution.

The next lemma is from \citep{birge:2004}. It can be used in the
derivation of risk lower bounds instead of the classical  Fano
Lemma. Just as Fano's Lemma, it states a lower bound on the error in
multiple hypothesis testing. However, as it only requires computing 
the Kullback-Leibler divergence to the localisation center, in the present setting, 
it significantly alleviates computations and makes the proof more concise and more transparent.

\begin{lem}(Birg\'e-Fano)
 \label{lem-fano}
Let  $P_0,\ldots,P_M$ be a collection of probability distributions on
some space, and let 
$A_0,\ldots,A_M$ be a collection of pairwise disjoint events, then the following holds
\begin{displaymath}
  \min_i P_i\{A_i\}  \leq \frac{2\mathe }{1+2\mathe} \vee \frac{\frac{1}{M}
    \sum_{i=1}^M \mathcal{K} (P_i, P_0)}{\ln (M+1)}  \, . 
\end{displaymath}
\end{lem}

In order to take advantage of Lemma \ref{lem-fano}, we use the
  Bayesian game designed in \citep{carpentierkim2014}. 

\begin{thm}\label{thm-kc-lower-bound:appendix}
Let $\gamma >0$, $\rho <-1$, and $0  \le v \le \mathe/(1+2\mathe)$. Then, for any tail index estimator $\widehat{\gamma}$ and any sample size $n$ such that $M=\lfloor \ln n\rfloor> \mathe^{1/v}$, there exists a collection $(P_i)_{i\leq M}$ of probability distributions such that
\begin{enumerate}[i)]
\item $P_i \in \emph{\textsf{MDA}}(\gamma_i)$ with $\gamma_i>\gamma$,
\item $P_i$ meets the von Mises condition with von Mises function $\eta_i$ satisfying
\begin{displaymath}
  \overline{\eta}_i(t) \leq \gamma t^{\rho_i}
\end{displaymath}
where $\rho_i =\rho +i/M<0$,
\item \begin{displaymath}
  \max_{i\leq M} P^{\otimes n}_i \left\{
    |\widehat{\gamma}-\gamma_i|\geq \frac{C_\rho}{4} \gamma_i
    \left(\frac{v\ln\ln n}{n}\right)^{|\rho_i|/(1+2|\rho_i|)} \right\} \geq \frac{1}{1+2\mathe} 
\end{displaymath}
and 
\begin{displaymath}
\max_{i\leq M}  \EXP_{P^{\otimes n}_i} \left[ \frac{|\widehat{\gamma}-\gamma_i|}{\gamma_i} \right] \geq \frac{C_\rho}{4(1+2\mathe)}  \left(\frac{v\ln\ln n}{n}\right)^{|\rho|/(1+2|\rho|)} \, , 
\end{displaymath}
with $C_\rho=1-\exp\left(-\tfrac{1}{2(1+2|\rho|)^2}\right)$. 
\end{enumerate}
\end{thm}

  \begin{proof}[Proof of Theorem \ref{thm-kc-lower-bound:appendix}]
    Choose $v$ so that $0\leq v\leq 2\mathe/(1+2\mathe)$. The number
    of alternative hypotheses $M$ is chosen in such a way that $ M/2 \leq \ln\left(
    n/(v\ln M)\right) \leq M$.  If  $\lfloor \ln n \rfloor
    \geq \mathe^{1/v}$, $M= \lfloor \ln n \rfloor$ will do.

    The center of localisation $P_0$ is the pure Pareto distribution
    with shape parameter $\gamma>0$ ($P_0\{(\tau,\infty)\} =
    \tau^{-1/\gamma}$). The local alternatives $P_1, \ldots, P_M$ are
    Pareto change point distributions.  Each $P_i$ is defined by a
    breakpoint $\tau_i>1$ and an ultimate Pareto index $\gamma_i$. 
If $F_i$ denotes the distribution function of $P_i$,
    \begin{displaymath}
      \overline{F}_i(x) =  x^{-1/\gamma} \mathbb{1}_{\{1\leq x \leq \tau_i\}} +
      \tau_i^{-1/\gamma} (x/\tau_i)^{-1/\gamma_i} \mathbb{1}_{\{x \geq \tau_i\}}  \,
      . 
    \end{displaymath}
    Karamata's representation of $(1/\overline{F}_i)^\leftarrow$ is
    \begin{displaymath}
      U_i(t) =  t^{\gamma_i}  \exp \left(\int_1^t \frac{\eta_i(s)}{s}\right) \mathrm{d}s
    \end{displaymath}
    with 
    \begin{math}
      \eta_i(s) = (\gamma-\gamma_i)\mathbb{1}_{\{s \le \tau_i^{1/\gamma}\}}\, . 
    \end{math}
    
    The Kullback-Leibler divergence between $P_i$ and $P_0$ is readily
    calculated,
    \begin{displaymath}
      \mathcal{K} (P_i,P_0) = \overline{F}_i(\tau_i)  \left(
        \frac{\gamma_i}{\gamma} -1 -\ln  \frac{\gamma_i}{\gamma}\right) =
      \tau_i^{-1/\gamma}  \left(
        \frac{\gamma_i}{\gamma} -1 -\ln  \frac{\gamma_i}{\gamma}\right) 
      \, . 
    \end{displaymath}
    If $\gamma_i>\gamma$, the next upper bound holds,
    \begin{displaymath}
      \mathcal{K} (P_i,P_0) \leq \frac{\tau_i^{-1/\gamma} }{2} \left(
        \frac{\gamma_i}{\gamma} -1 \right)^2  \, . 
    \end{displaymath}
    The breakpoints and tail indices are chosen in such a way that all
    upper bounds are equal (namely
    $n\tau_i^{-1/\gamma}(\gamma_i/\gamma-1)^2 $ does not depend on $i$),
    \begin{eqnarray*}
      \tau_i &= &
      \left( n/(v\ln M)\right)^{\gamma/(1+2|\rho_i|)} \\ 
      \gamma_i&=&\gamma +\gamma
      \left(n/(v \ln M ) \right)^{\rho_i/(1+2|\rho_i|)} \, , 
    \end{eqnarray*}
    so that $ \mathcal{K} (P_i^{\otimes n},P_0^{\otimes n}) = n
    \mathcal{K} (P_i,P_0) \leq { v\ln M}$,  for all $1\leq i\leq M$.
    
    Note that, for all $t>1$,
    \begin{displaymath}
      |\eta_i(t)|= |\gamma-\gamma_i|\mathbb{1}_{\{t\leq\tau_i^{1/\gamma}\}}\leq \gamma \tau_i^{\rho_i/\gamma}\mathbb{1}_{\{t\leq\tau_i^{1/\gamma}\}}\leq \gamma t^{\rho_i}
    \end{displaymath}
    the upper bound being achieved at $t=\tau_i^{1/\gamma} .$

 Now,    let $\widehat{\gamma}$ be any tail index estimator. Define region
    $A_i$, as the set of samples such that $\gamma_i$ minimises
    $|\widehat{\gamma}-\gamma_j|$, for $1\leq j\leq M$. Then, if the event $A_i$ is not realised,
$$|\widehat{\gamma}-\gamma_i|\geq \frac{1}{2} \min_{1\leq j\leq M,
  j\neq i} |\gamma_j-\gamma_i| \, .
$$
By Birg\'e's Lemma, 
\begin{displaymath}
\max_{i\leq M} \mathbb{P}_i^{\otimes n}\left\{ |\widehat{\gamma}-\gamma_i|\geq \frac{1}{2} \min_{1\leq j\leq M,
  j\neq i} |\gamma_j-\gamma_i|\right\} \geq \frac{1}{1+2\mathe} \, . 
\end{displaymath}
In order to make the whole construction useful, it remains to choose
the ``second-order parameters'' $\rho_i$'s (the true second-order parameter of each $P_i$ is infinite!). 
We will need an upper bound on
$\gamma_i/\gamma$ (but we already have $\gamma_i/\gamma\leq 2$), as
well as a lower bound on ${ |\gamma_j-\gamma_i|}/{\gamma}$ for $j\neq
i$ that scales like $\left({n}/{\ln \ln n}
\right)^{\rho_i/(1+2|\rho_i|)}$.

Following \citet{carpentierkim2014}, we finally choose $\rho_i$ as
$\rho_i = \rho + i/M$ for $1\leq i\leq M$.
Then, for $j < i$, using that $ M/2\le \ln(n/(v\ln M)) \le M$ and $\rho_i - \rho_j=(i-j)/M$, 
\begin{eqnarray*}
 \lefteqn{\frac{ |\gamma_j-\gamma_i|}{\gamma_i} }
 \\&\geq& \frac{ |\gamma_j-\gamma_i|}{2\gamma} \\
 & \geq & \frac{1}{2}\left(\frac{n}{v\ln M} \right)^{\rho_i/(1+2|\rho_i|)} \left| 1-  \left(\frac{n}{v\ln M} \right)^{\rho_j/(1+2|\rho_j|)-\rho_i/(1+2|\rho_i|)} \right| \\
 &\geq & \frac{1}{2}\left(\frac{n}{v\ln M} \right)^{\rho_i/(1+2|\rho_i|)} \left[ 1- \exp \left( \frac{-(i-j)}{M(1+2|\rho_i|)(1+2|\rho_j|)}\ln\left(\frac{n}{v\ln M}\right)\right)\right] \\
 &\geq& \frac{1}{2}\left(\frac{n}{v\ln M} \right)^{\rho_i/(1+2|\rho_i|)} \left[ 1- \exp \left( \frac{-(i-j)}{2(1+2|\rho_i|)(1+2|\rho_j|)}\right)\right] \\
 &\geq& \frac{C_\rho}{2}  \left(\frac{n}{v\ln M} \right)^{\rho_i/(1+2|\rho_i|)}
\end{eqnarray*}
where $C_\rho$ may be chosen as $1-\exp\left(-\tfrac{1}{2(1+2|\rho|)^2}\right)$.
\end{proof}

\end{document}